\documentclass{article}

\usepackage{arxiv_preprint,times}

\usepackage{graphicx}
\usepackage{float}
\usepackage{amsmath}
\usepackage{amssymb}
\usepackage{amsthm}
\usepackage{enumitem}
\usepackage{comment}
\usepackage{booktabs}
\usepackage{algorithm}
\usepackage[noend]{algpseudocode}
\usepackage{xcolor}
\usepackage{subcaption}
\usepackage{placeins}
\usepackage{hyperref}
\usepackage{url}
\usepackage{wrapfig}
\newtheorem{proposition}{Proposition}
\newtheorem{corollary}{Corollary}

\algrenewcommand\alglinenumber[1]{}

\title{
Learned Preconditioning for a Primal–Dual Interior-Point Method
}

\author{
Abhinav Madabhushi$^{1}$ \qquad
Jialin Liu$^{2}$ \qquad
Minxin Zhang$^{1}$\thanks{Corresponding author.}\\[0.5em]
{\normalfont $^{1}$Department of Mathematics, University of California, Los Angeles}\\
{\normalfont $^{2}$School of Data, Mathematical, and Statistical Sciences, University of Central Florida}\\[0.5em]
{\normalfont\small
\texttt{abhinavm@g.ucla.edu},
\texttt{jialin.liu@ucf.edu},
\texttt{zhangmx815@gmail.com}}
}

\begin{document}

\maketitle

\begin{abstract}
Interior-point methods (IPMs) are among the most widely used algorithms for constrained optimization, yet their Newton-based search directions require costly second-order information and large linear-system solves. Learning to optimize offers cheaper updates learned from data, but the singular behavior of logarithmic barriers near constraint boundaries makes IPMs highly sensitive to perturbations, complicating both warm starting and learning reliable updates. We introduce pdLIP, an IPM for smooth nonlinear programs that integrates learned preconditioning with pdProj, an all-shifted primal--dual projected-search IPM. A shared coordinate-wise recurrent network predicts a positive diagonal preconditioner that scales the right-hand side of the reduced Newton system for the primal step, and the remaining slack and multiplier directions are recovered analytically. The learned iterations avoid Hessian evaluations and Newton-system solves, using only first-order and coordinate-wise operations amenable to GPU parallelization. Training is self-supervised, with a loss based on a penalty--barrier merit function and the residual of perturbed optimality conditions, requiring neither target directions nor precomputed solutions. Primal and dual shifts mitigate the barrier's sensitivity to perturbations near constraint boundaries, enabling effective warm starting. Across four classes of 200-dimensional convex and nonconvex constrained problems, pdLIP warm starts reduce pdProj refinement iterations by 63--67\% compared with cold starts at the same KKT residual tolerance of $10^{-8}$, with negligible warm-start generation cost relative to the subsequent pdProj solve. Improvements persist on box-constrained QPs with $1000$ variables and extend to applications including portfolio optimization, support vector machines, and a nonlinear control example. These results demonstrate that learning only a structured scaling can substantially reduce the subsequent cost of high-accuracy constrained optimization.
\end{abstract}



\section{Introduction}
\vspace{-1.00ex}
Learning to optimize offers a way to reduce the cost of constrained
optimization by replacing expensive computations in classical algorithms
with learned components. For primal--dual interior-point methods (IPMs),
a major computational cost lies in obtaining Newton directions, which
requires second-order information and linear-system solves.
However, learning to
approximate Newton solves directly does not necessarily eliminate the cost of
constructing the system, including the evaluation of the
Hessian~\citep{gao2024ipm_lstm}. 
Logarithmic barriers pose a separate difficulty: their sensitivity
to perturbations near constraint boundaries complicates learning
reliable updates. These challenges motivate learned updates that
avoid Hessian evaluations and Newton-system solves while retaining
the analytical structure of a primal--dual IPM.

We introduce \emph{pdLIP}, a self-supervised primal--dual IPM for smooth
nonlinear programs. Rather than learning to solve a Newton system, pdLIP
learns only a positive diagonal preconditioner for the primal step.
It uses the all-shifted formulation of
pdProj~\citep{gill2024projected}, whose primal and dual shifts mitigate
barrier sensitivity near constraint boundaries and facilitate warm
starting. The learned iterations generate approximate primal--dual
solutions, which pdProj then refines to high accuracy. 
Our main contributions are:
\begin{itemize}[leftmargin=*,topsep=2pt,itemsep=2pt,parsep=0pt]

\item \textbf{Learned preconditioning.}
We replace solving a reduced KKT system with a learned positive
diagonal scaling while recovering the slack and multiplier directions
analytically. The learned iterations avoid Hessian evaluations and
Newton-system solves, using merely first-order and coordinate-wise operations
suited to GPU parallelization.

\item \textbf{Self-supervised training.}
A shared coordinate-wise recurrent network predicts the scaling, with
a parameter count independent of problem dimension. We train it using
a loss that combines the shifted penalty--barrier merit function with
the residual of perturbed optimality conditions, requiring neither
target directions nor precomputed solutions.

\item \textbf{Reduced refinement cost.}
Across four 200-variable convex and nonconvex benchmark classes, pdLIP
warm starts reduce pdProj refinement iterations by 63--67\% relative
to cold starts at the same KKT residual tolerance of \(10^{-8}\), with
negligible warm-start cost relative to refinement. Further experiments
show benefits on 1000-variable box-constrained QPs, portfolio
optimization, support vector machines, and a nonlinear quadrotor
control problem.
\end{itemize}

\vspace{-1.00ex}

\section{Related Works}
\vspace{-1.00ex}

\textbf{Learning-to-Optimize.} Learning-to-optimize (L2O) uses data from related problem instances to learn solution mappings or components of optimization algorithms \cite{chen2022learning2optimize}. Representative approaches include algorithm unrolling and direct solution prediction \cite{gregor2010learning,donti2021dc3,park2023self}, as well as recurrent learned optimizers \cite{andrychowicz2016learning,wichrowska2017learned}. Learning has also been used to replace specific components of classical algorithms, such as branching \cite{gasse2019exact}, cut selection \cite{tang2020reinforcement}, and pivot selection \cite{liu2024learningpivot}. For iterative optimization, \cite{liu2023towards} derive mathematical structures from fixed-point and convergence requirements to design learned update rules. \cite{luken2026selfsupervised} combine primal--dual prediction with learned iterative refinement using a KKT-based self-supervised loss. pdLIP also follows a structure-informed iterative approach, but learns only coordinate-wise scaling within an analytically specified IPM update, yielding effective primal--dual warm starts.

\vspace{-1.00ex}

\textbf{Learning-Augmented Interior-Point Methods.}
Neural networks have previously been combined with IPMs to generate warm starts \cite{fontova2007warm}. The most closely related modern approach is IPM-LSTM \cite{gao2024ipm_lstm}, which uses an LSTM to approximate the linear-system solve arising in an IPM and then warm-starts IPOPT. In contrast, pdLIP builds on the projected-search IPM of \cite{gill2024projected} and learns a positive coordinate-wise scaling of an IPM-derived direction, avoiding explicit construction of the full KKT Hessian while retaining the remaining primal--dual updates and projection analytically. 
OptNet embeds QP layers solved by a primal--dual IPM into neural networks \cite{amos2017optnet}, while differentiable convex optimization layers extend this approach to a broader class of convex programs \cite{agrawal2019differentiable}. These methods enable end-to-end training by differentiating optimization solution maps with respect to problem parameters, rather than learning the solver's iterative update rule.

\vspace{-0.75ex}
\section{Preliminaries}
\label{sec:preliminaries}
\vspace{-1.00ex}
We formulate the class of nonlinearly constrained optimization problems under consideration and review the projected-search primal–dual 
interior-point method that underlies our learned optimizer.

\vspace{-1.00ex}

\subsection{Problem Formulation}
\label{sec:problem-formulation}
\vspace{-1.00ex}
We consider smooth nonlinear programs of the form
\begin{equation}
\label{eq:general_nlp}
\min_{x\in\mathbb{R}^n} f(x)
\quad \text{s.t.} \quad
\ell^X \le x \le u^X, \quad
\ell^S \le c(x) \le u^S,
\end{equation}
where the objective function $f:\mathbb{R}^n\to\mathbb{R}$ and
the constraint mapping $c:\mathbb{R}^n\to\mathbb{R}^m$ are twice
continuously differentiable. The constant vectors $\ell^X,u^X$
and $\ell^S,u^S$ specify lower and upper bounds on $x$ and $c(x)$,
respectively. Infinite endpoints indicate absent bounds, while
equal finite bounds $\ell_i^S=u_i^S$ represent equality constraints.
We denote the constraint Jacobian by
$J(x)\in\mathbb{R}^{m\times n}$.

Introducing slack variables $s\in\mathbb{R}^m$, we equivalently
express the general constraints as
\begin{equation}
\label{eq:slack_reformulation}
c(x)-s=0,
\qquad
\ell^S\le s\le u^S.
\end{equation}
Define auxiliary variables
$x_1:=x-\ell^X$,
$x_2:=u^X-x,$
$s_1:=s-\ell^S$ and $ s_2:=u^S-s,$
with components associated with infinite bounds omitted. Let \(y\) denote the
vector of multipliers for \(c(x)-s=0\), and let \(z_1,z_2,w_1,w_2\) denote the
multipliers associated with the bounds on \(x\) and \(s\). 
The corresponding KKT conditions for the first-order optimality of \eqref{eq:general_nlp} are given by
\begin{equation}
\label{eq:kkt_conditions_main}
\begin{aligned}
\nabla f(x)-J(x)^\top y-z_1+z_2 &= 0, \\
y-w_1+w_2 &= 0, \\
c(x)-s &= 0, \\
a\ge0,\quad b\ge0,\quad a\odot b &=0,
\qquad (a,b)\in\mathcal{B},
\end{aligned}
\end{equation}
where $\mathcal{B}
:=
\{(x_1,z_1),(x_2,z_2),(s_1,w_1),(s_2,w_2)\},$ and \(\odot\) denotes componentwise multiplication.

\vspace{-1.50ex}

\subsection{pdProj: A Primal--Dual Projected-Search Interior-Point Method}
\label{sec:pdproj_background}
\vspace{-1.00ex}

Our method builds on the all-shifted  primal-dual projected-search
interior-point method pdProj~\citep{gill2024projected}, with detailed equations
derived in \citep{gill_zhang_equations}. We summarize only the
components needed to define the learned update.
\vspace{-1ex}
\paragraph{Perturbed optimality conditions and merit function.}
Let \(\mu^P,\mu^B>0\) be the penalty and barrier parameters, and let
\(\mathcal{E}\) collect an estimate \(y^E\) of \(y\) and estimates
\((a^E,b^E)\) of each pair \((a,b)\in\mathcal{B}\). For fixed
\((\mathcal{E},\mu^P,\mu^B)\), pdProj keeps the stationarity equations of
\eqref{eq:kkt_conditions_main} and replaces feasibility and complementarity
by the perturbed conditions
\begin{equation}
\label{eq:shifted_system}
c(x)-s=\mu^P(y^E-y),
\qquad
a\odot b=\mu^B(a^E-a)+\mu^B(b^E-b)
\quad\text{for all }(a,b)\in\mathcal{B}.
\end{equation}
Let \(v:=(x,s,y,z_1,z_2,w_1,w_2)\), and let the \emph{
path-following function} \(F(v)=F(v;\mathcal{E},\mu^P,\mu^B)\) stack the
stationarity residuals of \eqref{eq:kkt_conditions_main} with the residuals
of \eqref{eq:shifted_system} (Appendix~\ref{app:Formula for F}). 
To globalize the Newton iterations for \(F(v)=0\), pdProj uses an
\emph{all-shifted primal--dual penalty--barrier merit function}
\(M(v)=M(v;\mathcal{E},\mu^P,\mu^B)\),
which combines the objective function with shifted penalty terms for the equality 
constraints and shifted barrier terms for inequality and bound constraints
(Appendix~\ref{app:merit-function}). Let $e$ denote the vector of all $1$'s.
The domain of $M(v)$ is  the \emph{shifted interior}
\(\{(a,b)\in\mathcal B:a+\mu^Be>0,\ b+\mu^Be>0\}\), 
an open set containing the bound-feasible
region and its boundary. Thus, pdProj
can be initialized at a bound-feasible primal point with nonnegative
multipliers, and converges to a a primal--dual
solution of \eqref{eq:kkt_conditions_main} without the need of driving either
\(\mu^P\) or \(\mu^B>0\) to zero. The shifts also mitigates the ill-conditioning of 
the Newton equations. Together, these properties make pdProj particularly well suited both as a
framework for learning reliable updates and as a refinement solver for the
learned warm starts described in the following sections.

\vspace{-1.5ex}

\paragraph{Reduced Newton system.}
\vspace{-1.00ex}
The merit function $M(v)$ is constructed such that its gradient
$\nabla M(v)$ is a nonsingular linear transform of $F(v)$,
and the Newton system for minimizing $M$ approximates that for solving $F(v)=0$ \citep{gill2024projected}.
At \(v_k\) the search
direction solves
$H^M_k\,\Delta v_k=-\nabla M(v_k),$
where \(H^M_k\) is a positive definite approximation of \(\nabla^2M(v_k)\).
Eliminating the slack and bound-multiplier
components of \(\Delta v_k\) gives the reduced KKT system
\begin{equation}
\label{eq:reduced_newton_system}
\begin{pmatrix}
\widetilde H_k & J_k^\top\\
J_k & -D_k
\end{pmatrix}
\begin{pmatrix}
\Delta x_k\\
-\Delta y_k
\end{pmatrix}
=
-\begin{pmatrix}
r_{1,k}\\
r_{2,k}
\end{pmatrix},
\qquad
\widetilde H_k=\widehat H_k+D_{X,k},
\qquad
D_k=\mu^P_kI+D_{B,k},
\end{equation}
where \(J_k=J(x_k)\), $\widehat H_k$ is a positive definite approximation of
the exact Hessian $H(x_k,y_k)$,\(D_{X,k}\) and \(D_{B,k}\) are the positive diagonal
matrices induced by the shifted bounds on \(x\) and \(s\), and
\(r_{1,k}\), \(r_{2,k}\) are first-order residuals
(Appendix~\ref{app:newton-details}). In pdProj, the reduced system is solved via matrix factorization 
\cite[Algorithm~IC, p.~36]{WacB06}. Once \(\Delta x\) and \(\Delta y\) have been computed, 
the remaining components of the update $\Delta v$
are obtained via back substitution (Appendix~\ref{app:remaining-updates}).

\vspace{-1.00ex}

\paragraph{Projected search and global convergence.}
\vspace{-1.00ex}
Rather than truncating the step to remain strictly within the original
bounds as in conventional IPMs, pdProj projects each trial point onto
an enlargement of the bound-feasible region. This allows the search path
to change direction along the boundary of the enlarged region without
recomputing the search direction. Let \(v:=(x,s,y,z_1,z_2,w_1,w_2)\), 
and collect its shifted lower and upper
bounds in vectors $\ell$ and $u$.
With a fixed \emph{fraction-to-the-boundary} parameter \(\sigma\in(0,1)\),
for each iteration $k,$ define
\begin{equation}\label{eq:omega_k}
\Omega_k:=\bigl\{v:\min\{v_k-\sigma(v_k-\ell),\,\bar\ell\}\le v\le \max\{v_k+\sigma (u-v_k),\bar u\}\bigr\},
\end{equation} where $\bar \ell$ and $\bar u$ collect the original lower and upper bounds in 
\eqref{eq:kkt_conditions_main}. The box \(\Omega_k\) contains the
bound-feasible region and lies inside the shifted interior.
The next iterate is then obtained by \begin{equation}\label{eq:pdproj-update-main}
v_{k+1}=\operatorname{proj}_{\Omega_k}\bigl(v_k+\alpha_k\Delta v_k\bigr),
\end{equation}
where $\operatorname{proj}_{\Omega_k}$ denotes orthogonal projection onto $\Omega_k,$ 
and \(\alpha_k>0\) is a step size determined via a flexible 
\emph{quasi-Armijo search} along the
projected path \citep{gill2024projected,ferry2021projected}.

After this update, iteration \(k\) is classified as an O-, M-, or
F-iteration, and \((\mathcal{E}_k,\mu^P_k,\mu^B_k)\) are updated
accordingly. At an \emph{O-iteration}, \(v_{k+1}\) makes sufficient
progress toward satisfying \eqref{eq:kkt_conditions_main}, and the
estimates are updated using the current iterate. At an \emph{M-iteration},
\(v_{k+1}\) satisfies the approximate stationarity tests for \(M\),
and the estimates are updated with safeguards. The penalty parameter
\(\mu^P_k\) is halved if the feasibility test fails, and the barrier
parameter \(\mu^B_k\) is halved if the complementarity tests fail.
At an \emph{F-iteration}, the estimates and penalty and barrier
parameters remain unchanged. Additional safeguards maintain shifted
interiority when the barrier parameter is reduced.
Starting from any bound-feasible $v_0$, 
if \(\{H^M_k\}\) are uniformly
positive definite and bounded, then either infinitely many O-iterations
occur and every limit point of the O-iterates is a KKT point of \eqref{eq:general_nlp} 
under CAKKT regularity \citep{andreani2010new}; or, \eqref{eq:general_nlp} is an infeasible problem, and
every limit point of the M-iterates is an infeasible stationary point
\citep[Theorem~5.2]{gill2024projected}. 

\vspace{-1.00ex}

\section{Method}
\label{sec:method_theory}
\vspace{-1.00ex}

We now describe pdLIP, including its learned diagonal preconditioning, the coordinate-wise recurrent
parameterization, and the self-supervised training.

\vspace{-1.00ex}

\subsection{Motivation for the Learned Update}
\label{sec:learned_ipm_overview}
\vspace{-1.00ex}

Instead of solving the reduced KKT system \eqref{eq:reduced_newton_system}, we
further eliminating \(\Delta y_k\) in the system to derive
\begin{equation}
\label{eq:x_update_main}
\Delta x_k=-G_k^{-1}q_k,
\qquad
G_k:=\widetilde H_k+J_k^\top D_k^{-1}J_k,
\qquad
q_k:=r_{1,k}+J_k^\top D_k^{-1}r_{2,k},
\end{equation}
and $\Delta y_k$ is recovered by 
\begin{equation}
\label{eq:dual_recovery}
\Delta y_k=-D_k^{-1}\bigl(r_{2,k}+J_k\Delta x_k\bigr),
\end{equation}
where $D_k$ is the diagonal matrix given in \eqref{eq:reduced_newton_system}. 
Since \(D_k\) is diagonal, computing \(\Delta y_k\) for a given
\(\Delta x_k\) requires only a Jacobian--vector product followed by
componentwise division. Then
the remaining components of $\Delta v_k$ 
follow by back substitution as in pdProj. The vector \(q_k\) involves only the gradient
\(\nabla f(x_k)\), and matrix multiplications with \(J_k^\top\) and the diagonal matrix \(D_k^{-1}\); 
the operator \(G_k^{-1}\) is the only part that requires
second-order information and a matrix factorization.

This decomposition suggests a simple way to introduce learning into the
primal--dual update. Rather than approximating the entire Newton direction,
pdLIP keeps the reduced right-hand side \(q_k\) and replaces the expensive
application of \(G_k^{-1}\) with a learned positive diagonal scaling \(P_k\):
\begin{equation}
\label{eq:learned_update_overview}
\Delta x_k = -P_k q_k.
\end{equation}
The dual direction \(\Delta y_k\) is then recovered from
\eqref{eq:dual_recovery}, and the remaining slack and multiplier directions
are obtained by the same back-substitution formulas used in pdProj.
In this way, the learned model is used only to replace the costly reduced primal
solve, while the rest of the primal--dual direction is determined
analytically. The resulting learned iterations therefore avoid Hessian
evaluations and KKT matrix factorization without requiring the network to
predict the full primal--dual search direction.

\vspace{-1.00ex}

\subsection{LSTM Inputs}
\label{sec:lstm_inputs}
\vspace{-1.00ex}

At each iteration $k$, we construct the
feature vector for each coordinate \(j\in\{1,\ldots,n\}\) by
\begin{equation}
\label{eq:input_x}
\phi_{k,j}
:=
\Bigl(
x_j,\,
x_{1,j},\,
x_{2,j},\,
z_{1,j},\,
z_{2,j},\  ,
g_{x,j},\,
g_{z_1,j},\,
g_{z_2,j},\,
x^E_{1,j},\,
x^E_{2,j},\,
z^E_{1,j},\,
z^E_{2,j},\,
\log\mu^B,\,
\log\mu^P
\Bigr),
\end{equation}
where all quantities on the right-hand side are evaluated at iteration
\(k\), with the iteration index suppressed for readability. Here,
\(g_x:=\nabla_x M\), \(g_{z_1}:=\nabla_{z_1}M\), and
\(g_{z_2}:=\nabla_{z_2}M\), and
the superscript \(E\) denotes the current estimates as in pdProj.
All state, gradient, and estimate quantities in
\(\phi_{j}\) are evaluated at iteration \(k\), denoted as $\phi_{k,j}$.
The feature vector therefore combines the
primal--dual variables associated with \(x_j\), the
corresponding merit-gradient components and pdProj estimates, and the
current barrier and penalty parameters.

Because \(\mu^B\) and \(\mu^P\) may vary by several orders of magnitude
during the optimization, we use \(\log\mu^B\) and \(\log\mu^P\) as inputs rather
than the raw parameter values. Taking logarithms compresses
their range and keeps the input scales more consistent across iterations. Similar
log-scaled features have been used in prior learning-to-optimize methods
\citep{andrychowicz2016learning,wichrowska2017learned}.

\vspace{-1.00ex}

\paragraph{Coordinate-wise representation.}
\vspace{-1.00ex}
A single LSTM--MLP network, with parameters shared across all coordinates, processes
\(\phi_{k,j}\) for each $j$ and outputs one scaling coefficient for \(x_j\).
We include variables directly associated with
\(x_j\), together with their corresponding first-order information and estimates, 
without concatenating the full vectors \(s\), \(w\), and
\(y\).  Consequently, the number of
trainable parameters is independent of the problem dimension \(n\), and,
for a fixed network architecture, the cost of evaluating all coordinates
scales linearly with \(n\). Although the network processes coordinates
separately, its inputs are not purely local: \(g_{x,j}\) is a component
of the gradient of the full merit function \(M\) and therefore carries
information about the constraints and interactions with other variables.

\vspace{-1.00ex}

\subsection{LSTM Outputs and Search Direction}
\label{sec:lstm_outputs}
\vspace{-1.00ex}

Given \(\phi_{k,j}\), the shared LSTM followed by an MLP produces
\begin{equation}
\label{eq:mlp_lstm_abs}
\hat p_{k,j}
=
\mathrm{MLP}_{\theta_2}
\left(
\mathrm{LSTM}_{\theta_1}
\left(
\phi_{k,j},
h_{k-1,j}
\right)
\right),
\end{equation}
where \(\theta_1\) and \(\theta_2\) denote the trainable parameters of the
LSTM and MLP networks, respectively, and \(h_{k-1,j}\) denotes the recurrent hidden
state associated with coordinate \(j\). The recurrent state allows the
predicted scaling for each coordinate to depend on its optimization history,
rather than solely on its current features. The raw prediction is then
converted to a strictly positive scaling coefficient via
\begin{equation}
\label{eq:positive_lstm_output}
p_{k,j}
:=
\left|
\hat p_{k,j}
\right|
+\delta
\end{equation}
for a fixed $\delta>0.$
Stacking the predictions defines
$P_k
:=
\operatorname{diag}
\left(
p_{k,1},
\ldots,
p_{k,n}
\right)
\succ0,$
and the learned primal direction is
$\Delta x_k
=
-P_kq_k
=
-p_k\odot q_k.$ Once \(\Delta x_k\) is
determined by learned preconditioning, the remaining components of \(\Delta v_k\) are recovered
analytically using the same back-substitution equations as pdProj; the
full equations are included in Appendix~\ref{app:newton-details}.
We show that the resulting full direction
\(\Delta v_k\) is a descent direction for the penalty--barrier merit
function in
Appendix~\ref{app:positive-scaling}.

Finally, the trial point is projected onto the shifted feasible region
\(\Omega_k\) defined in \eqref{eq:omega_k}:
\[
v_{k+1}
=
\operatorname{proj}_{\Omega_k}\bigl(v_k+\Delta v_k\bigr).
\]
Unlike pdProj, which performs a projected search to determine the step
size \(\alpha_k\), pdLIP fixes \(\alpha_k=1\), with the magnitude of the update instead 
controlled by the learned diagonal scaling \(P_k\).
\vspace{-0.50ex}

\paragraph{Choice of output transformation.}
\vspace{-0.50ex}
We use the absolute-value transformation in
Eq.~\eqref{eq:positive_lstm_output} because it makes the scaling
coefficient depend on the magnitude rather than the sign of the raw
prediction, while ensuring a strictly positive lower bound
\(\delta>0\) without imposing an upper bound. The resulting
condition \(P_k\succ 0\) makes the effective reduced matrix \(P_k^{-1}\)
positive definite, consistent with the modified-Newton principle of
enforcing positive definiteness to obtain a descent direction
\citep{gill2024projected}. Unlike a conventional modified-Newton method,
pdLIP enforces this condition directly through the learned
parameterization, without forming or modifying a Hessian. An unscaled
sigmoid would also ensure positivity, but would restrict the coefficients
to a bounded interval. Softplus is positive and unbounded above, but maps
large negative predictions toward zero rather than preserving their
magnitude. In preliminary experiments, the absolute-value transformation
produced more reliable updates than these alternatives. This choice is
empirical: the descent property of the update requires the predicted coefficients to be
positive, but does not depend on the particular transformation used to
enforce positivity.

\vspace{-1.00ex}

\subsection{Self-supervised training}
\label{sec:self_supervised_loss}
\vspace{-1.00ex}
We train pdLIP in a self-supervised manner, without target Newton
directions or precomputed optimal solutions. At iteration \(k\), the
network predicts the positive diagonal scaling used to compute
\(\Delta x_k\). The remaining components of the primal--dual direction
are recovered analytically, and the training loss is evaluated at the
resulting iterate \(v_{k+1}\).

Let \(\mathcal{D}\) denote the current training batch and let \(\Theta=(\theta_1,\theta_2)\)
collect all trainable parameters of the LSTM--MLP network. For each problem instance
\(\xi\in\mathcal{D}\), define
\(
\zeta_k^{(\xi)}(\Theta)
:=
\left(
v_{k+1}^{(\xi)}(\Theta);
\mathcal{E}_k^{(\xi)},
\mu_k^{P,(\xi)},
\mu_k^{B,(\xi)}
\right).
\)
The batch loss is defined by
\begin{equation}
\label{eq:loss_function}
\mathcal{L}_k(\Theta)
=
\frac{1}{|\mathcal{D}|}
\sum_{\xi\in\mathcal{D}}
\left[
M\!\left(\zeta_k^{(\xi)}(\Theta)\right)
+
\log\!\left(
\left\|
F\!\left(\zeta_k^{(\xi)}(\Theta)\right)
\right\|_F^2
+
\epsilon
\right)
\right],
\end{equation}
where \(\epsilon>0\) is a small constant, \(M\) is the penalty-barrier merit function,
and \(F\) is the path-following residual introduced in
Section~\ref{sec:pdproj_background}, and $\|\cdot\|_F$ denotes the Frobenius norm.

\paragraph{Design of the training objective.}
The merit and residual terms serve complementary purposes in the training objective \eqref{eq:loss_function}. 
The merit term encourages each learned update to decrease the shifted penalty--barrier objective, while
the residual term acts as an auxiliary optimality penalty that guides the
new iterate toward the shifted primal--dual optimality system. The combined
objective therefore favors updates that reduce the merit function while
moving the iterate closer to satisfying the associated perturbed KKT
conditions.

The logarithm compresses large residual values so that the residual term
does not dominate the merit term, while \(\epsilon>0\) keeps the logarithm
finite as $F$ approaches $0$. Using the
optimization objective itself as a training signal is standard in
learning-to-optimize methods
\citep{andrychowicz2016learning,wichrowska2017learned,liu2023towards}.
The residual term serves as an auxiliary optimality regularizer, motivated
by self-supervised approaches based on first-order optimality conditions
\citep{luken2026selfsupervised}.

\paragraph{One-step unrolling.}
We use an unroll length of one. After computing \(v_{k+1}\), we evaluate
\(\mathcal{L}_k\), backpropagate through the learned update, and update
\(\Theta\). The iteration is then classified as an \(O\)-, \(M\)-, or
\(F\)-iteration using the original pdProj criteria, and the primal--dual
estimates and barrier and penalty parameters are updated accordingly
\citep{gill2024projected}. These updated quantities, together with
\(\nabla M(v_{k+1})\), are used to construct the input features for 
the next pdLIP iteration. Figure~\ref{fig:learned_projected_search_ipm}
summarizes one pdLIP iteration.

Longer unrolls allow errors in the learned scaling to accumulate through
the coupled primal--dual state, analytical back-substitution, and repeated
projections. They also increase memory requirements and susceptibility to
vanishing or exploding gradients. Moreover, because the \(O/M/F\) updates
may change the primal--dual estimates and barrier and penalty parameters,
successive learned steps may correspond to different shifted systems. Our
preliminary experiments indicate that unroll lengths greater than one were substantially
less stable. Further details are provided in
Appendix~\ref{app:unrolling_details}.

\begin{figure}[thp]
    \centering
\includegraphics[
  width=0.9\linewidth,
  height=0.2\textheight
]{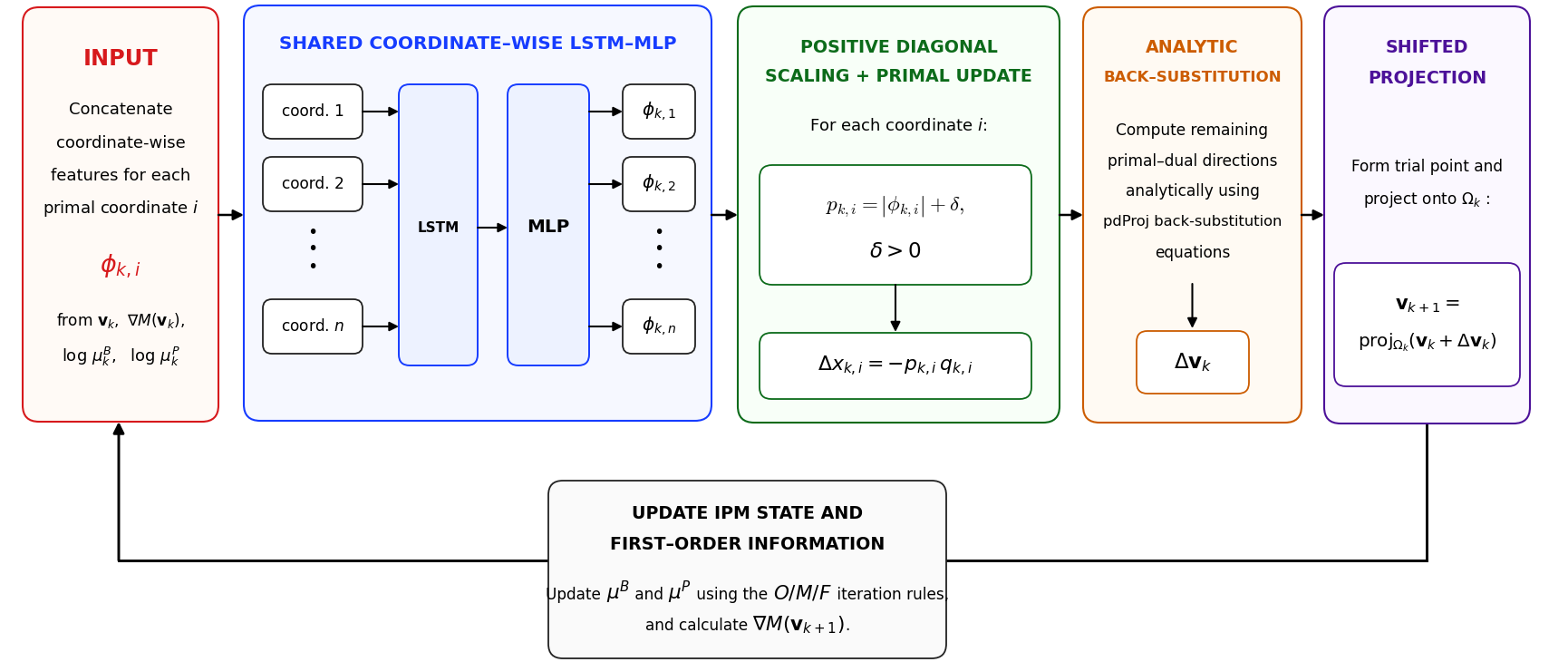}
   \caption{\textbf{Overview of one pdLIP iteration.}
At iteration \(k\), the shared LSTM--MLP maps each
\(\phi_{k,j}\) to a raw scaling \(\widehat p_{k,j}\).
The transformation \(p_{k,j}=|\widehat p_{k,j}|+\delta\) yields a
positive diagonal scaling \(P_k\succ0\), which defines the primal
direction \(\Delta x_k=-P_kq_k\). The remaining components of
\(\Delta v_k\) are recovered analytically from the pdProj equations,
after which the full trial step is projected onto \(\Omega_k\).
The estimates and barrier
and penalty parameters are then updated based on pdProj \(O/M/F\) rules.}
    \label{fig:learned_projected_search_ipm}
\end{figure}

\vspace{-1.00ex}


\section{Experiments}
\vspace{-1.00ex}
\subsection{Experimental Setup}
\label{sec:experimental_setup}
\vspace{-1.00ex}
For each problem class, we generate separate training, validation, and test
sets containing \(8{,}000\), \(1{,}000\), and \(1{,}000\) instances,
respectively, drawn from the same underlying distribution. The
coordinate-wise model consists of a single-layer LSTM with hidden dimension
\(128\), followed by an MLP head with architecture
\(128\rightarrow128\rightarrow1\): a linear layer of width \(128\), a ReLU
activation, and a scalar-output linear layer. The same LSTM--MLP parameters
are shared across all coordinates of \(x\).
Unless otherwise stated, we train the network using Adam with learning rate
\(10^{-4}\), batch size \(512\), \(50\) epochs, and \(300\) pdLIP iterations
per problem instance. Consistent with the one-step unrolling described in
Section~\ref{sec:self_supervised_loss}, the network parameters are updated
by backpropagation after each learned iteration. We set
\(\delta=10^{-8}\) in Eq.~\eqref{eq:positive_lstm_output} and
\(\epsilon=10^{-8}\) in Eq.~\eqref{eq:loss_function}. These architectural
and training settings are held fixed across problem classes unless otherwise
noted.

Because performance is more sensitive to the initial barrier and penalty
parameters than to the other hyperparameters considered, we tune
\((\mu_0^B,\mu_0^P)\) separately for each class of the problems based on 
validation performance. The selected values are then fixed for
evaluation on the held-out test set. Detailed results from the sensitivity
analysis are reported in Appendix~\ref{app:hyperparameter_tuning}.

We evaluate pdLIP as a primal--dual warm-start generator for
pdProj~\citep{gill2024projected}. pdProj is a natural refinement solver
because it accepts a full primal--dual initialization and shares the shifted
projected-search structure used by pdLIP. Among the iterates
\(v_0,\ldots,v_K\) generated by pdLIP, we select the warm start according to
\begin{equation}
\label{eq:warm_start_selection}
k^\star
=
\operatorname*{arg\,min}_{0\leq k\leq K}\chi_k,
\qquad
\chi_k
=
\chi_P(v_k)
+
\chi_D(v_k)
+
\chi_C(v_k,\mu_k^B),
\end{equation}
where $\chi_k$ is the same measure used in pdProj to assess progress toward optimality.
Its three components are
defined in Appendix~\ref{app:chi_measure}. The selected state
\(v_{k^\star}\) is used to initialize pdProj, while the barrier and penalty
parameters are initialized at their pdProj default values rather than
inherited from pdLIP. This avoids numerical-scale mismatches between the
learned and refinement phases.
We declare convergence for all solvers when the KKT residual satisfies
\(r_{\mathrm{KKT}}\leq10^{-8}\), with detailed equations of \(r_{\mathrm{KKT}}\) given in
Appendix~\ref{app:kkt_residual}. Both cold-started and warm-started pdProj converged on all reported test
instances. For comparison, we also report IPOPT iteration counts under the
same stopping criterion.

The learned warm start (inference) is generated on an NVIDIA RTX PRO 6000 Blackwell
Server Edition GPU using Google Colab, while pdProj is run on CPU.
All the training tasks are conducted on a workstation equipped with eight NVIDIA Quadro 
RTX 6000 GPUs (24 GiB of GPU memory each), two Intel Xeon Gold 5218 CPUs, and 502 GiB of system RAM.
In all the tables below, WS Cost denotes the time required to generate the
pdLIP warm start, and Total Time includes both WS Cost and the subsequent
pdProj refinement. All iteration counts and runtimes are averaged over the
test set. Reported reductions compare cold-started pdProj with the complete
pdLIP--pdProj pipeline. We use iteration reduction as the primary measure
of warm-start effectiveness because it more directly reflects the quality
of the initialization, whereas wall-clock time also depends on implementation
details and hardware. We do not directly compare pdProj and IPOPT runtimes because IPOPT
uses a faster linear solver than the current pdProj implementation,
giving it a per-iteration speed advantage that does not come from
the underlying IPM algorithm. 
Appendix~\ref{app:direct-approx} instead compares pdLIP directly with
IPOPT, demonstrating that pdLIP generates approximate solutions at lower
cost using GPU-friendly updates.

\vspace{-1.00ex}

\subsection{Warm-Start Performance}
\label{sec:experiments}
\vspace{-1.00ex}

\subsubsection{Convex and nonconvex benchmark problems}
\vspace{-1.00ex}
We first evaluate pdLIP on the constrained benchmarks used by
IPM-LSTM~\citep{gao2024ipm_lstm}. Each instance has \(200\) variables,
\(100\) equality constraints, and \(100\) inequality constraints. C-RHS
and C-ALL denote the convex QP settings, while NC-RHS and NC-ALL denote
their nonconvex counterparts. In the RHS setting, only the right-hand
sides of the equality constraints vary across instances; in the ALL
setting, all problem parameters are perturbed. Following
\citet{donti2021dc3}, the nonconvex variant replaces the linear term
\(p_0^\top x\) with \(p_0^\top\sin(x)\), giving the objective
\(
\frac{1}{2}x^\top Q_0x+p_0^\top\sin(x),
\)
where \(\sin(x)\) is applied componentwise.
IPM-LSTM uses IPOPT for refinement, whereas pdLIP uses pdProj. We therefore
evaluate warm-start effectiveness by the percentage reduction in refinement
iterations relative to each solver's own cold-start baseline. 
We therefore assess warm-start effectiveness by the percentage reduction in refinement iterations relative 
to each solver’s cold-start baseline, thereby focusing the comparison on the benefit of the learned 
initialization rather than differences between the refinement solvers.
Results are reported in Table~\ref{table:benchmark}.
Across the four benchmark classes, pdLIP warm starts reduce the number of
pdProj refinement iterations by \(63.38\%\)--\(67.10\%\) and total runtime
by \(62.81\%\)--\(68.29\%\) relative to cold-started pdProj. On every
shared benchmark class, the iteration reduction exceeds that reported for
IPM-LSTM by at least \(23\) percentage points.
\vspace{-1.00ex}

\begin{table}[H]
\centering
\caption{Warm-start results on the constrained benchmark problems from
\citep{gao2024ipm_lstm}.}\label{table:benchmark}
\label{tab:constrained_qp_full_comparison_200dim}

\begingroup
\small
\setlength{\tabcolsep}{3pt}
\begin{tabular*}{\textwidth}{@{\extracolsep{\fill}}lccccccc@{}}
\toprule
\textbf{Class}
& \textbf{IPOPT}
& \textbf{Cold pdProj}
& \multicolumn{3}{c}{\textbf{pdLIP + pdProj}}
& \multicolumn{2}{c}{\textbf{Reduction (Iter. / Time)}} \\
\cmidrule(lr){4-6}
\cmidrule(lr){7-8}

& \textbf{Iter.}
& \textbf{Iter. / Time}
& \textbf{WS Cost}
& \textbf{Iter. / Time}
& \textbf{Total Time}
& \textbf{pdLIP}
& \textbf{IPM-LSTM} \\
\midrule

C-RHS
& \(15.63\)
& \(13.41 / 7.53\) s
& \(20.11\) ms
& \(4.64 / 2.58\) s
& \(2.60\) s
& \(\mathbf{65.40 / 65.48}\%\)
& \(42.4 / 15.1\%\) \\

C-ALL
& \(15.82\)
& \(13.22 / 7.41\) s
& \(19.93\) ms
& \(4.35 / 2.42\) s
& \(2.44\) s
& \(\mathbf{67.10 / 67.11}\%\)
& \(35.7 / 7.5\%\) \\

NC-RHS
& \(15.60\)
& \(13.47 / 7.53\) s
& \(20.03\) ms
& \(4.93 / 2.78\) s
& \(2.80\) s
& \(\mathbf{63.38 / 62.81}\%\)
& \(27.5 / 1.3\%\) \\

NC-ALL
& \(15.72\)
& \(13.46 / 8.85\) s
& \(20.20\) ms
& \(4.84 / 2.79\) s
& \(2.81\) s
& \(\mathbf{64.04 / 68.29}\%\)
& \(15.4 / -8.6\%\) \\

\bottomrule
\end{tabular*}
\endgroup
\end{table}

\vspace{-1.00ex}

\paragraph{Scaling with Problem Dimension.}
\vspace{-1.00ex}
To assess scaling with problem dimension, we train and evaluate pdLIP on
convex box-constrained QPs with \(n\in\{10,200,1000\}\). Because the
\(n=10\) problems are substantially easier, we use \(100\) pdLIP iterations
per instance during training; all other settings remain unchanged.
As shown in Table~\ref{tab:box_qp_scaling}, pdLIP provides effective warm
starts for all three problem dimensions. At \(n=10\) and \(n=200\), it
reduces the number of pdProj refinement iterations by \(79.55\%\) and
\(80.81\%\), and total runtime by \(58.21\%\) and \(82.17\%\),
respectively. At \(n=1000\), the iteration reduction is smaller, at
\(47.90\%\), while total runtime is still reduced by \(66.73\%\).
Warm-start generation remains inexpensive at this scale, requiring only
\(40.94\) ms per instance, less than \(1\%\) of the \(4.53\) s required
for the subsequent pdProj refinement.
\vspace{-1.00ex}

\begin{table}[H]
\centering
\caption{Warm-start results on convex box-constrained QPs across dimensions.}
\label{tab:box_qp_scaling}

\begingroup
\small
\setlength{\tabcolsep}{4pt}
\begin{tabular*}{\textwidth}{@{\extracolsep{\fill}}lcccccc@{}}
\toprule
\textbf{Dim.}
& \textbf{IPOPT}
& \textbf{Cold pdProj}
& \multicolumn{3}{c}{\textbf{pdLIP + pdProj}}
& \textbf{Reduction} \\
\cmidrule(lr){4-6}

& \textbf{Iter.}
& \textbf{Iter. / Time}
& \textbf{WS Cost}
& \textbf{Iter. / Time}
& \textbf{Total Time}
& \textbf{Iter. / Time} \\
\midrule

\(n=10\)
& \(8.99\)
& \(7.45 / 6.66\) ms
& \(0.98\) ms
& \(1.52 / 1.80\) ms
& \(2.78\) ms
& \(\mathbf{79.55 / 58.21}\%\) \\

\(n=200\)
& \(13.48\)
& \(11.44 / 6.62\) s
& \(7.32\) ms
& \(2.19 / 1.17\) s
& \(1.18\) s
& \(\mathbf{80.81 / 82.17}\%\) \\

\(n=1000\)
& \(15.50\)
& \(13.87 / 13.74\) s
& \(40.94\) ms
& \(7.23 / 4.53\) s
& \(4.57\) s
& \(\mathbf{47.90 / 66.73}\%\) \\

\bottomrule
\end{tabular*}
\endgroup
\end{table}

\vspace{-1.00ex}

\subsubsection{Applications}

\paragraph{Portfolio Optimization and SVM Problems.}
\vspace{-1.00ex}
We next evaluate pdLIP on application-motivated portfolio optimization and
support vector machine (SVM) problems.We follow the instance-generation procedures of
\citet{chen2024expressive}, but consider a dense setting that removes the
exploitable sparsity present in the original benchmarks. The portfolio instances use randomly sampled
returns and a factor-model covariance matrix, while the SVM instances are
formulated as dense convex QPs. For the portfolio problems, \((s,t)\)
denotes the numbers of assets and factors; for the SVM problems, it denotes
the numbers of features and samples. Complete formulations and
instance-generation details are provided in
Appendices~\ref{app:portfolio_formulation} and
\ref{app:svm_formulation}.
As reported in Table~\ref{tab:pdproj_warm_start_portfolio_svm},
across the portfolio optimizationand SVM settings, pdLIP warm starts reduce the number
of pdProj refinement iterations by \(22.40\%\)--\(61.28\%\) and total
runtime by \(20.71\%\)--\(61.65\%\). The gains are more modest on the larger portfolio setting, 
but pdLIP still
reduces pdProj refinement iterations by \(22.40\%\) and total runtime by
\(20.71\%\).
\vspace{-1.00ex}

\vspace{-1.00ex}

\begin{table}[H]
\centering
\caption{Warm-start results on portfolio optimization and SVM problems.}
\label{tab:pdproj_warm_start_portfolio_svm}
\begingroup
\small
\setlength{\tabcolsep}{3pt}
\begin{tabular*}{\textwidth}{@{\extracolsep{\fill}}lccccccc@{}}
\toprule
\textbf{Problem}
& \textbf{\((s,t)\)}
& \textbf{IPOPT}
& \textbf{Cold pdProj}
& \multicolumn{3}{c}{\textbf{pdLIP + pdProj}}
& \textbf{Reduction} \\
\cmidrule(lr){5-7}

&
& \textbf{Iter.}
& \textbf{Iter. / Time}
& \textbf{WS Cost}
& \textbf{Iter. / Time}
& \textbf{Total Time}
& \textbf{Iter. / Time} \\
\midrule

Portfolio
& \((50,5)\)
& \(14.68\)
& \(8.99 / 22.94\) ms
& \(5.16\) ms
& \(3.48 / 7.82\) ms
& \(12.98\) ms
& \(\mathbf{61.28 / 43.43}\%\) \\

Portfolio
& \((200,20)\)
& \(21.06\)
& \(12.07 / 6.50\) s
& \(14.07\) ms
& \(9.36 / 5.14\) s
& \(5.15\) s
& \(\mathbf{22.40 / 20.71}\%\) \\

SVM
& \((5,50)\)
& \(12.52\)
& \(9.30 / 3.45\) s
& \(11.57\) ms
& \(4.52 / 1.66\) s
& \(1.67\) s
& \(\mathbf{51.39 / 51.55}\%\) \\

SVM
& \((20,200)\)
& \(15.46\)
& \(13.83 / 10.21\) s
& \(35.95\) ms
& \(6.31 / 3.88\) s
& \(3.92\) s
& \(\mathbf{54.42 / 61.65}\%\) \\

\bottomrule
\end{tabular*}
\endgroup
\end{table}

\vspace{-1.00ex}

\paragraph{Quadrotor Navigation.}
\vspace{-1.00ex}
We further evaluate pdLIP on the single-quadrotor navigation problem
of~\citet{viljoen2026scaling}. The decision variables comprise the
quadrotor state trajectory and rotor inputs, which are optimized to steer
the vehicle from a sampled initial state toward the goal while satisfying
the discretized dynamics and avoiding an obstacle. We use a \(200\)-variable
formulation with a quadratic objective, \(46\) linear equality constraints,
\(110\) nonlinear equality constraints arising from the dynamics, \(11\)
nonlinear obstacle-avoidance inequalities, and variable bounds. In the
planar visualization, the quadrotor is represented by a point at its
horizontal position, and collision avoidance requires this point to remain
outside a disk centered at the obstacle, with clearance radius equal to the
sum of the obstacle and quadrotor radii. The nonlinear, nonconvex dynamics
and constraints make this formulation structurally more
complex than the QP benchmarks considered above. The complete formulation
and instance-generation procedure are in
Appendix~\ref{app:quadrotor_navigation}.

\vspace{-1.00ex}

\begin{table}[H]
\centering
\caption{Warm-start results on the \(200\)-dimensional quadrotor
optimization problem.}
\label{tab:quadrotor_warm_start}

\begingroup
\small
\setlength{\tabcolsep}{3pt}
\begin{tabular*}{\textwidth}{@{\extracolsep{\fill}}lcccccc@{}}
\toprule
\textbf{Problem}
& \textbf{IPOPT}
& \textbf{Cold pdProj}
& \multicolumn{3}{c}{\textbf{pdLIP + pdProj}}
& \textbf{Reduction} \\
\cmidrule(lr){4-6}

& \textbf{Iter.}
& \textbf{Iter. / Time}
& \textbf{WS Cost}
& \textbf{Iter. / Time}
& \textbf{Total Time}
& \textbf{Iter. / Time} \\
\midrule

Quadrotor
& \(9.20\)
& \(9.74 / 19.05\) s
& \(71.99\) ms
& \(8.00 / 15.74\) s
& \(15.81\) s
& \(\mathbf{17.89 / 17.00}\%\) \\

\bottomrule
\end{tabular*}
\endgroup
\end{table}

\par\noindent
\begin{minipage}{\linewidth}
\begin{minipage}[t]{0.57\linewidth}
    \vspace{0pt}
    As reported in Table~\ref{tab:quadrotor_warm_start}, warm-starting
    pdProj with pdLIP reduces the average number of refinement iterations
    from \(9.74\) to \(8.00\), a reduction of \(17.89\%\), and reduces
    total runtime by \(17.00\%\).
    Figure~\ref{fig:quadrotor_trajectory} compares the trajectory generated
    by pdLIP with the high-accuracy trajectory obtained by pdProj.
    Although pdLIP satisfies the nonlinear dynamics only approximately,
    it yields a similar obstacle-avoiding trajectory and a useful
    primal--dual warm start for pdProj. These results extend the evaluation of pdLIP beyond QP
    benchmarks and show that it can provide effective warm starts for a
    nonlinear, nonconvex optimal control problem.
    Additional solver-level statistics, selected \(n=50\) benchmark results,
and direct approximate-solution comparisons without pdProj refinement
are provided in Appendix~\ref{app:additional-experiments}.
\end{minipage}\hfill%
\begin{minipage}[t]{0.40\linewidth}
    \vspace{0pt}
    \centering
    \includegraphics[width=\linewidth]
    {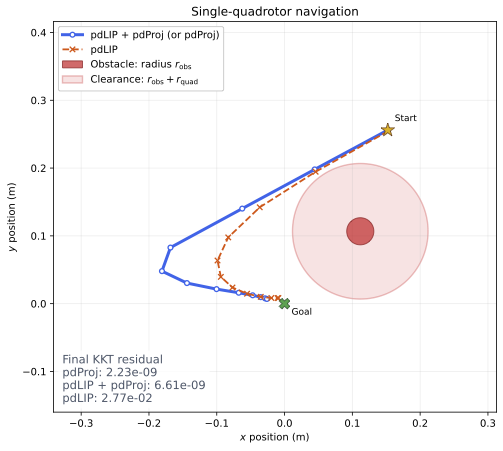}
    \captionof{figure}{Single-quadrotor trajectories generated by pdLIP
    and pdProj.}
    \label{fig:quadrotor_trajectory}
\end{minipage}
\end{minipage}
\par

\vspace{-1.00ex}

\section{Conclusions and Future work}
\vspace{-1.00ex}
We introduced pdLIP, a primal-dual interior-point method that replaces
the Newton update with a positive diagonal scaling learned
via self-supervised training, while recovering the slack and
multiplier directions from the pdProj equations. The learned iterations
avoid evaluating Hessians or solving Newton systems, and primal and dual
shifts reduce sensitivity to perturbations near constraint boundaries,
facilitating both learning and warm starting. On four 200-variable convex
and nonconvex benchmark classes, pdLIP warm starts reduce pdProj refinement
iterations by 63--67\% at the same KKT residual tolerance of \(10^{-8}\),
with negligible warm-start cost relative to refinement. Benefits also
extend to larger box-constrained QPs, portfolio optimization, SVMs, and
nonlinear control. These results show that learning only a diagonal
preconditioner can provide useful approximate solutions and reduce the
work needed by a conventional solver to reach high accuracy. Future work 
includes exploring alternative training losses and sequence models
to improve warm-start quality on more challenging nonlinear and nonconvex
problems, and investigating whether a single trained model can provide effective
warm starts across different problem sizes and distributions.

\vspace{-1.00ex}

\subsection*{AI use statement}
\vspace{-1.00ex}

Generative AI tools, primarily ChatGPT, were used to assist with drafting and
editing portions of the manuscript, LaTeX formatting, code debugging and
editing, and presentation of results. All AI-assisted content was reviewed
and verified by the authors. The authors made the final decisions regarding
the methodology, implementation, experiments, and conclusions and take full
responsibility for the contents of this work.

\vspace{-1.00ex}

\subsection*{Reproducibility Statement}
\vspace{-1.00ex}

The main paper and Appendix provide the problem formulations, data-generation
procedures, model architecture, training protocol, hyperparameter-selection
procedure, solver parameters, stopping criteria, and evaluation methodology used
in our experiments. These details are provided to enable reproduction of all
reported experimental results.



\bibliography{references}
\bibliographystyle{references}


\appendix

\section{Appendix}


\subsection{Additional Details for the Projected-Search Interior-Point System}

\subsubsection{Full Shifted Penalty--Barrier Merit Function}
\label{app:merit-function}

The projected-search interior-point method of
\citet{gill2024projected} defines a shifted primal--dual
penalty--barrier merit function, with the corresponding equations given in
\citet{gill_zhang_equations}. In the notation used in the main text, this
function is
\begin{equation}
\label{eq:merit_extended_appendix}
\begin{aligned}
&M(x,x_1,x_2,s,s_1,s_2,y,\nu,z_1,z_2,w_1,w_2;
\mu^P,\mu^A,\mu^B,
y^E,\nu^E,z_1^E,z_2^E,w_1^E,w_2^E)
\\[0.4em]
&=
f(x)
-(c(x)-s)^\top y^E
+\frac{1}{2\mu^P}\|c(x)-s\|^2
+\frac{1}{2\mu^P}\|c(x)-s+\mu^P(y-y^E)\|^2
\\
&\quad
-(Ax-b)^\top \nu^E
+\frac{1}{2\mu^A}\|Ax-b\|^2
+\frac{1}{2\mu^A}\|Ax-b+\mu^A(\nu-\nu^E)\|^2
\\[0.5em]
&\quad
-\sum_{j=1}^{n_L}
\Bigl\{
\mu^B\bigl([z_1^E]_j+[x_1^E]_j+\mu^B\bigr)
\ln\!\left([z_1+\mu^B e]_j[x_1+\mu^B e]_j^2\right)
\\
&\hspace{8em}
-[z_1\odot(x_1+\mu^B e)]_j
-2\mu^B[x_1]_j
\Bigr\}
\\
&\quad
-\sum_{j=1}^{n_U}
\Bigl\{
\mu^B\bigl([z_2^E]_j+[x_2^E]_j+\mu^B\bigr)
\ln\!\left([z_2+\mu^B e]_j[x_2+\mu^B e]_j^2\right)
\\
&\hspace{8em}
-[z_2\odot(x_2+\mu^B e)]_j
-2\mu^B[x_2]_j
\Bigr\}
\\
&\quad
-\sum_{i=1}^{m_L}
\Bigl\{
\mu^B\bigl([w_1^E]_i+[s_1^E]_i+\mu^B\bigr)
\ln\!\left([w_1+\mu^B e]_i[s_1+\mu^B e]_i^2\right)
\\
&\hspace{8em}
-[w_1\odot(s_1+\mu^B e)]_i
-2\mu^B[s_1]_i
\Bigr\}
\\
&\quad
-\sum_{i=1}^{m_U}
\Bigl\{
\mu^B\bigl([w_2^E]_i+[s_2^E]_i+\mu^B\bigr)
\ln\!\left([w_2+\mu^B e]_i[s_2+\mu^B e]_i^2\right)
\\
&\hspace{8em}
-[w_2\odot(s_2+\mu^B e)]_i
-2\mu^B[s_2]_i
\Bigr\}.
\end{aligned}
\end{equation}

The original nonlinear constraints are represented by
\[
c(x)-s=0.
\]
The additional term \(Ax-b=0\) is not part of the original problem
statement. It is an auxiliary equality system introduced inside the
projected-search method when the current iterate becomes infeasible after
reducing \(\mu^B\) or \(\mu^P\). The variable \(\nu\) is the multiplier
associated with the auxiliary feasibility-restoration constraint
\(Ax-b=0\). The parameters \(\mu^P\), \(\mu^A\), and \(\mu^B\) denote the
corresponding penalty and barrier parameters.

The auxiliary bound variables satisfy
\begin{equation}
\label{eq:auxiliary_constraints_appendix}
\begin{aligned}
x-x_1&=\ell^X,
&
x+x_2&=u^X,
\\
s-s_1&=\ell^S,
&
s+s_2&=u^S.
\end{aligned}
\end{equation}
The shifted positivity conditions are
\begin{equation}
\label{eq:shifted_positivity_appendix}
\begin{aligned}
x_1+\mu^B e&>0,
&
z_1+\mu^B e&>0,
\\
x_2+\mu^B e&>0,
&
z_2+\mu^B e&>0,
\\
s_1+\mu^B e&>0,
&
w_1+\mu^B e&>0,
\\
s_2+\mu^B e&>0,
&
w_2+\mu^B e&>0.
\end{aligned}
\end{equation}

\subsubsection{Full Newton System and Definitions}
\label{app:newton-details}

After linear transformations, the Newton step can be written in the block form
\begin{equation}
\label{eq:full_newton_system_appendix}
\resizebox{\textwidth}{!}{%
$
\left(
\begin{array}{cccccccc}
D_A & 0 & 0 & 0 & 0 & A & 0 & 0 \\[0.2em]
0 & D^{z}_{1} & 0 & 0 & 0 & I & 0 & 0 \\[0.2em]
0 & 0 & D^{z}_{2} & 0 & 0 & -I & 0 & 0 \\[0.2em]
0 & 0 & 0 & D^{w}_{1} & 0 & 0 & 0 & 0 \\[0.2em]
0 & 0 & 0 & 0 & D^{w}_{2} & 0 & 0 & 0 \\[0.2em]
-A^{\top} & -I & I & 0 & 0 & \widehat H & -J^{\top} & 0 \\[0.2em]
0 & 0 & 0 & 0 & 0 & J & 0 & D_P
\end{array}
\right)
\!
\left(
\begin{array}{c}
\Delta \nu \\
\Delta z_{1} \\
\Delta z_{2} \\
\Delta w_{1} \\
\Delta w_{2} \\
\Delta x \\
\Delta s \\
\Delta y
\end{array}
\right)
=
-
\left(
\begin{array}{c}
D_A(\nu-\pi^\nu) \\
D^{z}_{1}(z_{1}-\pi_{1}^{z}) \\
D^{z}_{2}(z_{2}-\pi_{2}^{z}) \\
D^{w}_{1}(w_{1}-\pi_{1}^{w}) \\
D^{w}_{2}(w_{2}-\pi_{2}^{w}) \\
g-J^{\top}y-A^{\top}\nu-z_{1}+z_{2} \\
D_P(y-\pi^{Y})
\end{array}
\right).
$
}
\end{equation}

Here, \(g=\nabla_x f(x)\), and \(J\) denotes the Jacobian of the nonlinear
constraint mapping \(c(x)\). The matrix $\widehat H$ is a positive definite approximation of
the exact Hessian $H(x,y)$. The matrix \(A\)
denotes the Jacobian of the auxiliary equality system \(Ax-b=0\), which is
used only during the feasibility-restoration step caused by reductions in
\(\mu^B\) or \(\mu^P\). Further details on the construction of these
quantities are given in \citep{gill_zhang_equations}.

The shifted diagonal matrices are
\begin{equation}
\label{eq:shifted_diagonal_matrices_appendix}
\begin{aligned}
X_1^\mu&=\operatorname{diag}(x_1+\mu^B e),
&
Z_1^\mu&=\operatorname{diag}(z_1+\mu^B e),
\\
X_2^\mu&=\operatorname{diag}(x_2+\mu^B e),
&
Z_2^\mu&=\operatorname{diag}(z_2+\mu^B e),
\\
S_1^\mu&=\operatorname{diag}(s_1+\mu^B e),
&
W_1^\mu&=\operatorname{diag}(w_1+\mu^B e),
\\
S_2^\mu&=\operatorname{diag}(s_2+\mu^B e),
&
W_2^\mu&=\operatorname{diag}(w_2+\mu^B e).
\end{aligned}
\end{equation}

These matrices are used to define the bound-related diagonal matrices and
shifted targets:
\begin{equation}
\label{eq:bound_targets_appendix}
\begin{aligned}
D_1^z&=X_1^\mu(Z_1^\mu)^{-1},
&
\pi_1^z
&=
\mu^B(X_1^\mu)^{-1}
\bigl(z_1^E-x_1+x_1^E\bigr),
\\[0.4em]
D_2^z&=X_2^\mu(Z_2^\mu)^{-1},
&
\pi_2^z
&=
\mu^B(X_2^\mu)^{-1}
\bigl(z_2^E-x_2+x_2^E\bigr),
\\[0.4em]
D_1^w&=S_1^\mu(W_1^\mu)^{-1},
&
\pi_1^w
&=
\mu^B(S_1^\mu)^{-1}
\bigl(w_1^E-s_1+s_1^E\bigr),
\\[0.4em]
D_2^w&=S_2^\mu(W_2^\mu)^{-1},
&
\pi_2^w
&=
\mu^B(S_2^\mu)^{-1}
\bigl(w_2^E-s_2+s_2^E\bigr).
\end{aligned}
\end{equation}

The combined matrices are defined as:
\begin{equation}
\label{eq:combined_matrices_appendix}
D_Z=D_1^z+D_2^z,
\qquad
D_B=D_1^w+D_2^w,
\end{equation}
and the combined shifted multiplier targets are:
\begin{equation}
\label{eq:combined_pi_appendix}
\pi^z=\pi_1^z-\pi_2^z,
\qquad
\pi^w=(D_B)^{-1}
\bigl(D_1^w\pi_1^w+D_2^w\pi_2^w\bigr).
\end{equation}

The penalty-related matrices and shifted targets are:
\begin{equation}
\label{eq:penalty_targets_appendix}
\begin{aligned}
D_P&=\mu^P I_m,
&
\pi^Y&=y^E-\frac{1}{\mu^P}(c(x)-s),
\\[0.4em]
D_A&=\mu^A I_A,
&
\pi^\nu&=\nu^E-\frac{1}{\mu^A}(Ax-b).
\end{aligned}
\end{equation}

\subsubsection{Derivation of the Reduced Newton System}
\label{app:reduced-system-derivation}

Starting from the full Newton system in
\eqref{eq:full_newton_system_appendix}, the variables
\(\Delta\nu,\Delta z_1,\Delta z_2,\Delta w_1,\Delta w_2,\Delta s\)
can be eliminated to obtain a reduced system in \(\Delta x\) and
\(\Delta y\):
\begin{equation}
\label{eq:reduced_newton_system_appendix}
\begin{pmatrix}
\widetilde H & -J^\top \\[0.3em]
J & D_P+D_B
\end{pmatrix}
\begin{pmatrix}
\Delta x \\[0.3em]
\Delta y
\end{pmatrix}
=
-
\begin{pmatrix}
r_1 \\[0.3em]
r_2
\end{pmatrix},
\end{equation}
where
\begin{equation}
\label{eq:residuals_appendix}
r_1=g-J^\top y-A^\top\pi^\nu-\pi^z,
\qquad
r_2=D_B(y-\pi^w)+D_P(y-\pi^Y).
\end{equation}

\begin{equation}
\label{eq:modified_primal_hessian_appendix}
\widetilde H=H^{B}+A^\top D_A^{-1}A+D_Z.
\end{equation}

Solving the second block row for \(\Delta y\) gives
\begin{equation}
\label{eq:delta_y_appendix}
\Delta y
=
-(D_P+D_B)^{-1}(r_2+J\Delta x).
\end{equation}
Substituting this expression into the first block row yields
\begin{equation}
\label{eq:delta_x_appendix}
\Delta x
=
\bigl(\widetilde H+J^\top(D_P+D_B)^{-1}J\bigr)^{-1}
\Bigl(
-r_1-J^\top(D_P+D_B)^{-1}r_2
\Bigr).
\end{equation}

This is the expression used in the main text to identify the
curvature-dependent Newton component replaced by the learned model.

\subsubsection{Remaining Slack and Multiplier Updates}
\label{app:remaining-updates}

Once \(\Delta x\) and \(\Delta y\) have been computed, the remaining updates
are obtained by back substitution. In particular,
\begin{equation}
\label{eq:delta_s_appendix}
\Delta s=-D_B(y+\Delta y-\pi^w).
\end{equation}
The multiplier updates are
\begin{equation}
\label{eq:remaining_updates_appendix}
\begin{aligned}
\Delta w_1
&=
-(S_1^\mu)^{-1}
\Bigl(
w_1\odot(s+\Delta s-\ell^S+\mu^B e)
-\mu^B w_1^E
+\mu^B(s-s^E+\Delta s)
\Bigr),
\\[0.5em]
\Delta w_2
&=
-(S_2^\mu)^{-1}
\Bigl(
w_2\odot(u^S-(s+\Delta s)+\mu^B e)
-\mu^B w_2^E
+\mu^B(s^E-s-\Delta s)
\Bigr),
\\[0.5em]
\Delta z_1
&=
-(X_1^\mu)^{-1}
\Bigl(
z_1\odot(x+\Delta x-\ell^X+\mu^B e)
-\mu^B z_1^E
+\mu^B(x-x^E+\Delta x)
\Bigr),
\\[0.5em]
\Delta z_2
&=
-(X_2^\mu)^{-1}
\Bigl(
z_2\odot(u^X-(x+\Delta x)+\mu^B e)
-\mu^B z_2^E
+\mu^B(x^E-x-\Delta x)
\Bigr).
\end{aligned}
\end{equation}



\subsubsection{Shifted Path-Following Residual Used in Training}
\label{app:Formula for F}

The training objective uses the same shifted path-following residual \(F\)
introduced in Section~\eqref{sec:pdproj_background}. Suppressing fixed
arguments, the residual stacks the stationarity, shifted feasibility,
restoration, bound, and shifted complementarity residual blocks:
\begin{equation}
\label{eq:F_normalized_gradient}
F(v)
=
\begin{pmatrix}
\nabla f(x)-J(x)^\top y-A^\top\nu-z_1+z_2 \\[4pt]
y-w_1+w_2 \\[4pt]
c(x)-s+\mu^P(y-y^E) \\[4pt]
Ax-b+\mu^A(\nu-\nu^E) \\[4pt]
E_xx-b_x \\[4pt]
L_xs-h_x \\[6pt]
z_1\odot x_1
+\mu^B(z_1-z_1^E)
+\mu^B(x_1-x_1^E) \\[6pt]
z_2\odot x_2
+\mu^B(z_2-z_2^E)
+\mu^B(x_2-x_2^E) \\[6pt]
w_1\odot s_1
+\mu^B(w_1-w_1^E)
+\mu^B(s_1-s_1^E) \\[6pt]
w_2\odot s_2
+\mu^B(w_2-w_2^E)
+\mu^B(s_2-s_2^E)
\end{pmatrix}.
\end{equation}
The rows involving \(A\) and the corresponding restoration variables are
omitted when the restoration system is inactive. Likewise, components
associated with missing bounds are omitted.

The residual term in the training objective uses the squared Frobenius norm,
\begin{equation}
\label{eq:F_frobenius_norm}
\|F(v)\|_F^2
=
\sum_{i,j} F_{ij}(v)^2.
\end{equation}


\subsection{Theoretical Details}

\subsubsection{Positive Diagonal Scaling}
\label{app:positive-scaling}

\begin{proposition}
\label{prop:positive_scaling_descent}
Let \(q_k\) be defined in~\eqref{eq:x_update_main} and
\(P_k\succ0\) be the learned diagonal preconditioner.
Suppose the learned primal direction is
\[
\Delta x_k=-P_kq_k,
\]
and the remaining components of the primal--dual direction
\(\Delta v_k\) are obtained from the unchanged pdProj
back-substitution equations. Then, whenever
\(\nabla M(v_k)\neq0\), the resulting full direction
\(\Delta v_k\) is a descent direction for the penalty--barrier merit
function \(M\); that is,
\[
\nabla M(v_k)^\top\Delta v_k<0.
\]
\end{proposition}

\begin{proof}
Since the learned scaling is defined by
\[
p_{k,j}
=
|\hat p_{k,j}|+\delta,
\qquad
\delta>0,
\]
we have \(p_{k,j}>0\) for all \(j\). Hence,
\[
P_k\succ0
\qquad\text{and}\qquad
P_k^{-1}\succ0.
\]
The learned update
\[
\Delta x_k=-P_kq_k
\]
is therefore equivalently written as
\[
P_k^{-1}\Delta x_k=-q_k.
\]

The reduced projected-search system of
\citep{gill2024projected} has the form
\[
\left(
\widetilde H_k+J_k^\top D_k^{-1}J_k
\right)\Delta x_k=-q_k.
\]
Thus, the learned update corresponds to the implicit choice
\[
\widetilde H_k^{\mathrm L}
=
P_k^{-1}
-
J_k^\top D_k^{-1}J_k,
\]
for which
\[
\widetilde H_k^{\mathrm L}
+
J_k^\top D_k^{-1}J_k
=
P_k^{-1}\succ0.
\]

By the positive-definiteness relation established in
\citep{gill2024projected}, the corresponding approximate merit Hessian
\(H_k^{M,\mathrm L}\) is therefore positive definite. Since the remaining
components of \(\Delta v_k\) are obtained from the unchanged
back-substitution equations, the full direction satisfies
\[
H_k^{M,\mathrm L}\Delta v_k
=
-\nabla M(v_k).
\]
Consequently,
\[
\nabla M(v_k)^\top\Delta v_k
=
-\nabla M(v_k)^\top
\left(H_k^{M,\mathrm L}\right)^{-1}
\nabla M(v_k)
<0
\]
whenever \(\nabla M(v_k)\neq0\). Hence, \(\Delta v_k\) is a descent
direction for the all-shifted penalty--barrier merit function.
\end{proof}

\subsubsection{Local Consistency of the Augmented Training Objective}
\label{app:augmented-loss-proof}

Fix the pdProj estimates \(\mathcal{E}\) and parameters
\(\mu^P\) and \(\mu^B\). Define the augmented loss
\begin{equation}
\label{eq:augmented_loss_appendix}
\ell_\epsilon(v)
=
M(v)
+
\log\left(
\|F(v)\|_F^2+\epsilon
\right),
\qquad
\epsilon>0.
\end{equation}

\begin{proposition}
\label{prop:augmented_loss_consistency}
Let \(v^\star\) lie in the domain of the shifted penalty--barrier merit
function \(M\). Suppose that, in a neighborhood of \(v^\star\), the
residual--merit relationship
\[
F(v)=U(v)\nabla M(v)
\]
holds, where \(U(v)\) is nonsingular, and suppose that \(F\) is
differentiable at \(v^\star\). Then:

\begin{enumerate}
    \item If \(\nabla M(v^\star)=0\), then
    \[
    F(v^\star)=0
    \qquad\text{and}\qquad
    \nabla \ell_\epsilon(v^\star)=0.
    \]
    Thus, every stationary point of \(M\) in the shifted interior remains
    a stationary point of \(\ell_\epsilon\).

    \item If, in addition, \(v^\star\) is a local minimizer of \(M\), then
    \(v^\star\) is also a local minimizer of \(\ell_\epsilon\).
\end{enumerate}
\end{proposition}

\begin{corollary}
\label{cor:augmented_loss_convergence}
Under the conditions of
Proposition~\ref{prop:augmented_loss_consistency}, suppose additionally
that \(v^\star\) is a local minimizer of \(M\) and that \(U(v)^{-1}\) is
locally bounded near \(v^\star\). If a sequence
\(\{v_j\}\) in a neighborhood of \(v^\star\) satisfies
\[
\ell_\epsilon(v_j)
\rightarrow
\ell_\epsilon(v^\star),
\]
then
\[
M(v_j)\rightarrow M(v^\star),
\qquad
\|F(v_j)\|_F\rightarrow0,
\qquad
\|\nabla M(v_j)\|\rightarrow0.
\]
\end{corollary}

\paragraph{Proof of Proposition~\ref{prop:augmented_loss_consistency}.}

\begin{proof}
Suppose first that
\[
\nabla M(v^\star)=0.
\]
The residual--merit relationship established in
\citep{gill2024projected} gives
\[
F(v)=U(v)\nabla M(v),
\]
where \(U(v)\) is nonsingular in the shifted interior. Hence,
\[
F(v^\star)=0.
\]

Define
\[
R_\epsilon(v)
=
\log\left(
\|F(v)\|_F^2+\epsilon
\right),
\]
so that
\[
\ell_\epsilon(v)=M(v)+R_\epsilon(v).
\]
By the chain rule,
\[
\nabla R_\epsilon(v)
=
\frac{
2J_F(v)^\top F(v)
}{
\|F(v)\|_F^2+\epsilon
},
\]
where \(J_F(v)\) denotes the Jacobian of \(F\), with \(F\) viewed as a
vectorized residual. Since \(F(v^\star)=0\),
\[
\nabla R_\epsilon(v^\star)=0.
\]
Therefore,
\[
\nabla \ell_\epsilon(v^\star)
=
\nabla M(v^\star)
+
\nabla R_\epsilon(v^\star)
=
0.
\]
This proves the first statement.

Now suppose additionally that \(v^\star\) is a local minimizer of \(M\).
Then there exists a neighborhood \(\mathcal{N}\) of \(v^\star\) such that
\[
M(v)\geq M(v^\star),
\qquad
v\in\mathcal{N}.
\]
From the first part,
\[
F(v^\star)=0.
\]
Therefore, for every \(v\in\mathcal{N}\),
\begin{align}
\ell_\epsilon(v)-\ell_\epsilon(v^\star)
&=
M(v)-M(v^\star)
+
\log\left(
\frac{\|F(v)\|_F^2+\epsilon}
{\|F(v^\star)\|_F^2+\epsilon}
\right)
\nonumber\\
&=
M(v)-M(v^\star)
+
\log\left(
1+\frac{\|F(v)\|_F^2}{\epsilon}
\right).
\end{align}
Both terms on the right-hand side are nonnegative. Consequently,
\[
\ell_\epsilon(v)-\ell_\epsilon(v^\star)\geq0,
\]
and hence \(v^\star\) is a local minimizer of
\(\ell_\epsilon\).
\end{proof}

\paragraph{Proof of Corollary~\ref{cor:augmented_loss_convergence}.}

\begin{proof}
Let \(\{v_j\}\subset\mathcal{N}\) satisfy
\[
\ell_\epsilon(v_j)
\rightarrow
\ell_\epsilon(v^\star).
\]
Since \(v^\star\) is a local minimizer of \(M\), the proof of
Proposition~\ref{prop:augmented_loss_consistency} gives
\[
\ell_\epsilon(v_j)-\ell_\epsilon(v^\star)
=
M(v_j)-M(v^\star)
+
\log\left(
1+\frac{\|F(v_j)\|_F^2}{\epsilon}
\right),
\]
where both terms on the right-hand side are nonnegative. Hence each term
must converge to zero. Therefore,
\[
M(v_j)\rightarrow M(v^\star)
\]
and
\[
\log\left(
1+\frac{\|F(v_j)\|_F^2}{\epsilon}
\right)
\rightarrow0,
\]
which implies
\[
\|F(v_j)\|_F^2\rightarrow0,
\]
and hence
\[
\|F(v_j)\|_F\rightarrow0.
\]

By the residual--merit relationship of
\citep{gill2024projected},
\[
F(v)=U(v)\nabla M(v).
\]
Since \(U(v)^{-1}\) is locally bounded near \(v^\star\),
\[
\nabla M(v_j)
=
U(v_j)^{-1}F(v_j),
\]
and therefore
\[
\|\nabla M(v_j)\|\rightarrow0.
\]
\end{proof}


\subsection{Training and Evaluation Details}

\subsubsection{One-Step Training Protocol}
\label{app:unrolling_details}

pdLIP uses an unroll length of one. At IPM iteration \(k\), the current
primal--dual state and associated features are passed through the shared
LSTM--MLP to obtain the learned scaling \(P_k\). The resulting primal
direction is combined with the analytical pdProj back-substitution equations
and projection to obtain \(v_{k+1}\). The loss
\(\mathcal{L}_k\) in~\eqref{eq:loss_function} is then evaluated and
immediately backpropagated before proceeding to the next IPM iteration.

We experimented with longer unrolls in which \(K>1\) learned pdProj search
steps were taken before applying the \(O/M/F\) logic. In this setting, errors
in the learned scaling are propagated through the coupled primal--dual state,
analytical back-substitution, and repeated projection across several
successive steps. These errors can therefore accumulate throughout the
unroll, while the longer computational graph also increases memory
requirements and susceptibility to vanishing or exploding gradients. In
preliminary experiments, this formulation resulted in substantially less
stable training.

We also considered applying the \(O/M/F\) logic after every learned step
while delaying backpropagation across multiple iterations. This was similarly
unstable because the \(O/M/F\) logic can modify the primal--dual estimates and
the barrier and penalty parameters. Consequently, successive learned steps
may correspond to different shifted merit and residual systems, so the gradients used for backpropagation are computed across a sequence of
iterations in which the underlying shifted optimization problem may change.

Based on these observations, we evaluate and backpropagate
\(\mathcal{L}_k\) immediately after each learned pdLIP update and then apply
the pdProj \(O/M/F\) logic before constructing the next iteration.

\subsubsection{Common KKT Residual}
\label{app:kkt_residual}

We use a solver-neutral KKT residual to measure convergence across the
learned optimizer, IPOPT, and pdProj. Following the notation of the main
text, let \(y\) denote the Lagrange multiplier associated with
\(c(x)-s=0\), and let \(z_1,z_2,w_1,w_2\) denote the multipliers associated
with the bound constraints on \(x\) and \(s\).

The stationarity residuals with respect to \(x\) and \(s\) are
\[
r_x
=
\nabla f(x)-J(x)^\top y-z_1+z_2,
\]
and
\[
r_s
=
y-w_1+w_2,
\]
where \(J(x)\) is the Jacobian of \(c(x)\). We use the scaled stationarity
residual
\[
\widehat r_{\mathrm{stat}}
=
\max\left\{
\frac{\|r_x\|_\infty}
{\max\{1,\|\nabla f(x)\|_\infty\}},
\;
\|r_s\|_\infty
\right\}.
\]

The common KKT residual is defined as
\[
\begin{aligned}
r_{\mathrm{KKT}}
=
\max\Big\{&
\widehat r_{\mathrm{stat}},
\ \|c(x)-s\|_\infty,\\
&
\|\min(x-\ell^X,0)\|_\infty,
\ \|\min(u^X-x,0)\|_\infty,\\
&
\|\min(s-\ell^S,0)\|_\infty,
\ \|\min(u^S-s,0)\|_\infty,\\
&
\|z_1\odot(x-\ell^X)\|_\infty,
\ \|z_2\odot(u^X-x)\|_\infty,\\
&
\|w_1\odot(s-\ell^S)\|_\infty,
\ \|w_2\odot(u^S-s)\|_\infty,\\
&
\|\min(z_1,0)\|_\infty,
\ \|\min(z_2,0)\|_\infty,
\ \|\min(w_1,0)\|_\infty,
\ \|\min(w_2,0)\|_\infty
\Big\}.
\end{aligned}
\]
Terms associated with missing constraints or infinite bounds are omitted.
A problem is counted as converged when
\(r_{\mathrm{KKT}}\leq\tau\), where \(\tau\) is the tolerance specified
for the corresponding experiment.

\subsubsection{The \texorpdfstring{$\chi$}{chi} Measure}
\label{app:chi_measure}

Following \citep[Section~5.1]{gill2024projected}, we use the progress measure
\begin{equation}
\chi(v,\mu^B)
=
\chi_P(v)
+
\chi_D(v)
+
\chi_C(v,\mu^B),
\end{equation}
where \(\chi_P\), \(\chi_D\), and \(\chi_C\) measure primal feasibility,
dual stationarity, and complementarity, respectively. This measure is used
for best-iterate selection and in the O-, M-, and F-iteration update logic.

Recall the bound slacks
\[
x_1=x-\ell^X,
\qquad
x_2=u^X-x,
\qquad
s_1=s-\ell^S,
\qquad
s_2=u^S-s,
\]
and the set of finite primal--dual bound pairs
\[
\mathcal{B}
=
\{(x_1,z_1),(x_2,z_2),(s_1,w_1),(s_2,w_2)\}.
\]
Components corresponding to infinite bounds are omitted. All vector minima,
maxima, absolute values, and products below are interpreted componentwise.

\paragraph{Primal feasibility.}
We define
\begin{equation}
\label{eq:chi_primal}
\chi_P(v)
=
\left\|
\begin{pmatrix}
\dfrac{c(x)-s}{\max\{1,\|s\|_2\}} \\[8pt]
\min\{x_1,0\} \\[4pt]
\min\{x_2,0\} \\[4pt]
\min\{s_1,0\} \\[4pt]
\min\{s_2,0\}
\end{pmatrix}
\right\|_2.
\end{equation}

\paragraph{Dual stationarity.}
Define
\begin{equation}
\sigma_D
=
\max\left\{
1,
\|\nabla f(x)\|_2,
\|J(x)^\top y\|_2,
\|z_1\|_2,
\|z_2\|_2
\right\}.
\end{equation}
Then
\begin{equation}
\label{eq:chi_dual}
\chi_D(v)
=
\max\left\{
\frac{
\|\nabla f(x)-J(x)^\top y-z_1+z_2\|_2
}{\sigma_D},
\;
\|y-w_1+w_2\|_2
\right\}.
\end{equation}

\paragraph{Complementarity.}
For each finite primal--dual bound pair
\((a,b)\in\mathcal{B}\), define
\begin{align}
q_1(a,b)
&=
\max\left\{
|\min\{a,b,0\}|,
|a\odot b|
\right\},
\\
q_2(a,b,\mu^B)
&=
\max\Bigl\{
\mu^B e,\,
|\min\{a+\mu^B e,b+\mu^B e,0\}|,\notag\\
&\hspace{7em}
|(a+\mu^B e)\odot(b+\mu^B e)|
\Bigr\},
\end{align}
where \(e\) is the vector of ones of the appropriate dimension. The
complementarity contribution associated with \((a,b)\) is
\[
q(a,b,\mu^B)
=
\min\{q_1(a,b),q_2(a,b,\mu^B)\}.
\]
Stacking these contributions over all finite bound pairs gives
\begin{equation}
\label{eq:chi_complementarity}
\chi_C(v,\mu^B)
=
\frac{1}{\max\{1,\|\nabla f(x)\|_2\}}
\left\|
\operatorname{col}_{(a,b)\in\mathcal{B}}
q(a,b,\mu^B)
\right\|_2.
\end{equation}

\subsubsection{Primal and Dual Infeasibility Measures}
\label{app:infeasibility_measures}

We also use the primal and dual infeasibility measures defined in
\citep{gill2024projected}. Let
\begin{equation}
\label{eq:sigma_appendix}
\sigma_{\mathrm D}
:=
\max\!\left\{
1,\;
\|\nabla f(x)\|_\infty,\;
\|J(x)^\top y\|_\infty,\;
\|A^\top \nu\|_\infty,\;
\|z_1\|_\infty,\;
\|z_2\|_\infty
\right\}.
\end{equation}
The primal infeasibility is
\begin{equation}
\label{eq:primal_infeasibility_appendix}
e_P(x,s)
=
\left\|
\begin{pmatrix}
\dfrac{c(x)-s}{\max\{1,\|s\|_\infty\}} \\[8pt]
\min\{0,s\} \\[6pt]
\dfrac{Ax-b}{\max\{1,\|x\|_\infty\}} \\[10pt]
\dfrac{\min\{0,x-\ell\}}{\max\{1,\|x\|_\infty\}} \\[10pt]
\dfrac{\min\{0,u-x\}}{\max\{1,\|x\|_\infty\}}
\end{pmatrix}
\right\|_\infty.
\end{equation}
The dual infeasibility is
\begin{equation}
\label{eq:dual_infeasibility_appendix}
e_D(x,s,y,\nu,z_1,z_2,w)
=
\left\|
\begin{pmatrix}
\dfrac{\nabla f(x)-J(x)^\top y-A^\top\nu-z_1+z_2}
{\sigma_{\mathrm D}} \\[10pt]
\|w-y\|_\infty \\[6pt]
w\odot\min\{1,s\} \\[6pt]
z_1\odot\min\{1,x-\ell\} \\[6pt]
z_2\odot\min\{1,u-x\}
\end{pmatrix}
\right\|_\infty.
\end{equation}

\subsubsection{Hyperparameter Selection and Sensitivity}
\label{app:hyperparameter_tuning}

For each candidate configuration, we evaluate the mean primal and dual
infeasibility measures \(e_P\) and \(e_D\) and rank the configurations using
the combined validation criterion
\begin{equation}
\label{eq:combined_validation_metric}
e_{\mathrm{comb}}
=
\sqrt{e_P^2+e_D^2}.
\end{equation}
Lower values indicate better validation performance, and we select the
hyperparameter configuration with the lowest \(e_{\mathrm{comb}}\).
For the \(50\)-dimensional constrained QP sensitivity experiments below,
training uses \(200\) IPM iterations; the remaining sensitivity experiments
use \(300\) IPM iterations.


\begin{table}[H]
\centering
\caption{Hyperparameter sensitivity for convex QP RHS,
\(n=50,\;m_{\mathrm{eq}}=25,\;m_{\mathrm{ineq}}=25\).}
\label{tab:hparam_qp_convex_rhs_50}
\begin{tabular}{cccccc}
\toprule
\textbf{Rank} & \(\boldsymbol{\mu_B^0}\) & \(\boldsymbol{\mu_P^0}\) & \textbf{Mean \(e_P\)} & \textbf{Mean \(e_D\)} & \textbf{$e_{\mathrm{comb}}$} \\
\midrule
\(1\) & \(10^{-4}\) & \(10^{4}\) & \(0.0071\) & \(0.0027\) & \(0.0076\) \\
\(2\) & \(0.1\) & \(10^{3}\) & \(0.0013\) & \(0.0090\) & \(0.0091\) \\
\(3\) & \(0.1\) & \(5\times10^{3}\) & \(0.0134\) & \(0.0024\) & \(0.0136\) \\
\(4\) & \(10^{-3}\) & \(10^{3}\) & \(0.0092\) & \(0.0133\) & \(0.0162\) \\
\(5\) & \(0.01\) & \(10^{3}\) & \(0.0074\) & \(0.0211\) & \(0.0223\) \\
\bottomrule
\end{tabular}
\end{table}

\begin{table}[H]
\centering
\caption{Hyperparameter sensitivity for convex QP ALL,
\(n=50,\;m_{\mathrm{eq}}=25,\;m_{\mathrm{ineq}}=25\).}
\label{tab:hparam_qp_convex_all_50}
\begin{tabular}{cccccc}
\toprule
\textbf{Rank} & \(\boldsymbol{\mu_B^0}\) & \(\boldsymbol{\mu_P^0}\) & \textbf{Mean \(e_P\)} & \textbf{Mean \(e_D\)} & \textbf{$e_{\mathrm{comb}}$} \\
\midrule
\(1\) & \(0.01\) & \(5\times10^{3}\) & \(0.0013\) & \(0.0018\) & \(0.0022\) \\
\(2\) & \(10^{-3}\) & \(5\times10^{3}\) & \(0.0328\) & \(0.0034\) & \(0.0330\) \\
\(3\) & \(10^{-3}\) & \(10^{3}\) & \(0.0335\) & \(0.0373\) & \(0.0501\) \\
\(4\) & \(0.1\) & \(5\times10^{3}\) & \(0.0487\) & \(0.0142\) & \(0.0507\) \\
\(5\) & \(0.01\) & \(10^{3}\) & \(0.1082\) & \(0.0278\) & \(0.1117\) \\
\bottomrule
\end{tabular}
\end{table}

\begin{table}[H]
\centering
\caption{Hyperparameter sensitivity for nonconvex QP RHS,
\(n=50,\;m_{\mathrm{eq}}=25,\;m_{\mathrm{ineq}}=25\).}
\label{tab:hparam_qp_nonconvex_rhs_50}
\begin{tabular}{cccccc}
\toprule
\textbf{Rank} & \(\boldsymbol{\mu_B^0}\) & \(\boldsymbol{\mu_P^0}\) & \textbf{Mean \(e_P\)} & \textbf{Mean \(e_D\)} & \textbf{$e_{\mathrm{comb}}$} \\
\midrule
\(1\) & \(0.1\) & \(5\times10^{3}\) & \(0.0072\) & \(0.0030\) & \(0.0078\) \\
\(2\) & \(0.1\) & \(10^{3}\) & \(0.0014\) & \(0.0086\) & \(0.0087\) \\
\(3\) & \(10^{-3}\) & \(10^{3}\) & \(0.0083\) & \(0.0091\) & \(0.0123\) \\
\(4\) & \(0.01\) & \(5\times10^{3}\) & \(0.0326\) & \(0.0035\) & \(0.0327\) \\
\(5\) & \(0.01\) & \(10^{4}\) & \(0.0507\) & \(0.0065\) & \(0.0511\) \\
\bottomrule
\end{tabular}
\end{table}

\begin{table}[H]
\centering
\caption{Hyperparameter sensitivity for nonconvex QP ALL,
\(n=50,\;m_{\mathrm{eq}}=25,\;m_{\mathrm{ineq}}=25\).}
\label{tab:hparam_qp_nonconvex_all_50}
\begin{tabular}{cccccc}
\toprule
\textbf{Rank} & \(\boldsymbol{\mu_B^0}\) & \(\boldsymbol{\mu_P^0}\) & \textbf{Mean \(e_P\)} & \textbf{Mean \(e_D\)} & \textbf{$e_{\mathrm{comb}}$} \\
\midrule
\(1\) & \(10^{-3}\) & \(10^{3}\) & \(0.0062\) & \(0.0119\) & \(0.0134\) \\
\(2\) & \(10^{-3}\) & \(5\times10^{3}\) & \(0.0119\) & \(0.0079\) & \(0.0143\) \\
\(3\) & \(0.1\) & \(5\times10^{3}\) & \(0.0135\) & \(0.0074\) & \(0.0154\) \\
\(4\) & \(0.01\) & \(5\times10^{3}\) & \(0.0154\) & \(0.0073\) & \(0.0171\) \\
\(5\) & \(0.1\) & \(10^{4}\) & \(0.0183\) & \(0.0073\) & \(0.0197\) \\
\bottomrule
\end{tabular}
\end{table}

\begin{table}[H]
\centering
\caption{Hyperparameter sensitivity for convex QP RHS,
\(n=200,\;m_{\mathrm{eq}}=100,\;m_{\mathrm{ineq}}=100\).}
\label{tab:hparam_qp_convex_rhs_200}
\begin{tabular}{cccccc}
\toprule
\textbf{Rank} & \(\boldsymbol{\mu_B^0}\) & \(\boldsymbol{\mu_P^0}\) & \textbf{Mean \(e_P\)} & \textbf{Mean \(e_D\)} & \textbf{$e_{\mathrm{comb}}$} \\
\midrule
\(1\) & \(0.01\) & \(10^{4}\) & \(0.0084\) & \(0.0061\) & \(0.0103\) \\
\(2\) & \(10^{-3}\) & \(10^{4}\) & \(0.0112\) & \(0.0062\) & \(0.0128\) \\
\(3\) & \(0.01\) & \(5\times10^{3}\) & \(0.0101\) & \(0.0101\) & \(0.0142\) \\
\(4\) & \(10^{-3}\) & \(10^{3}\) & \(0.0198\) & \(0.0368\) & \(0.0418\) \\
\(5\) & \(0.1\) & \(10^{4}\) & \(0.0471\) & \(0.0068\) & \(0.0476\) \\
\bottomrule
\end{tabular}
\end{table}

\begin{table}[H]
\centering
\caption{Hyperparameter sensitivity for convex QP ALL,
\(n=200,\;m_{\mathrm{eq}}=100,\;m_{\mathrm{ineq}}=100\).}
\label{tab:hparam_qp_convex_all_200}
\begin{tabular}{cccccc}
\toprule
\textbf{Rank} & \(\boldsymbol{\mu_B^0}\) & \(\boldsymbol{\mu_P^0}\) & \textbf{Mean \(e_P\)} & \textbf{Mean \(e_D\)} & \textbf{$e_{\mathrm{comb}}$} \\
\midrule
\(1\) & \(10^{-3}\) & \(10^{4}\) & \(0.0057\) & \(0.0057\) & \(0.0080\) \\
\(2\) & \(0.01\) & \(10^{4}\) & \(0.0069\) & \(0.0059\) & \(0.0091\) \\
\(3\) & \(10^{-3}\) & \(5\times10^{3}\) & \(0.0042\) & \(0.0096\) & \(0.0105\) \\
\(4\) & \(0.01\) & \(5\times10^{3}\) & \(0.0208\) & \(0.0106\) & \(0.0233\) \\
\(5\) & \(0.01\) & \(10^{3}\) & \(0.0053\) & \(0.0350\) & \(0.0354\) \\
\bottomrule
\end{tabular}
\end{table}

\begin{table}[H]
\centering
\caption{Hyperparameter sensitivity for nonconvex QP RHS,
\(n=200,\;m_{\mathrm{eq}}=100,\;m_{\mathrm{ineq}}=100\).}
\label{tab:hparam_qp_nonconvex_rhs_200}
\begin{tabular}{cccccc}
\toprule
\textbf{Rank} & \(\boldsymbol{\mu_B^0}\) & \(\boldsymbol{\mu_P^0}\) & \textbf{Mean \(e_P\)} & \textbf{Mean \(e_D\)} & \textbf{$e_{\mathrm{comb}}$} \\
\midrule
\(1\) & \(0.01\) & \(5\times10^{3}\) & \(0.0211\) & \(0.0152\) & \(0.0260\) \\
\(2\) & \(0.1\) & \(5\times10^{3}\) & \(0.0203\) & \(0.0166\) & \(0.0262\) \\
\(3\) & \(0.1\) & \(10^{4}\) & \(0.0340\) & \(0.0110\) & \(0.0358\) \\
\(4\) & \(10^{-3}\) & \(5\times10^{3}\) & \(0.0502\) & \(0.0186\) & \(0.0535\) \\
\(5\) & \(10^{-3}\) & \(10^{3}\) & \(0.0205\) & \(0.0512\) & \(0.0551\) \\
\bottomrule
\end{tabular}
\end{table}

\begin{table}[H]
\centering
\caption{Hyperparameter sensitivity for nonconvex QP ALL,
\(n=200,\;m_{\mathrm{eq}}=100,\;m_{\mathrm{ineq}}=100\).}
\label{tab:hparam_qp_nonconvex_all_200}
\begin{tabular}{cccccc}
\toprule
\textbf{Rank} & \(\boldsymbol{\mu_B^0}\) & \(\boldsymbol{\mu_P^0}\) & \textbf{Mean \(e_P\)} & \textbf{Mean \(e_D\)} & \textbf{$e_{\mathrm{comb}}$} \\
\midrule
\(1\) & \(0.01\) & \(5\times10^{3}\) & \(0.0214\) & \(0.0236\) & \(0.0318\) \\
\(2\) & \(10^{-3}\) & \(5\times10^{3}\) & \(0.0303\) & \(0.0220\) & \(0.0375\) \\
\(3\) & \(0.01\) & \(10^{4}\) & \(0.0400\) & \(0.0138\) & \(0.0423\) \\
\(4\) & \(0.1\) & \(10^{4}\) & \(0.0408\) & \(0.0152\) & \(0.0436\) \\
\(5\) & \(10^{-3}\) & \(10^{3}\) & \(0.0152\) & \(0.0508\) & \(0.0530\) \\
\bottomrule
\end{tabular}
\end{table}

\begin{table}[H]
\centering
\caption{Hyperparameter sensitivity for portfolio problems, \(s=50,\;t=5\).}
\label{tab:hparam_portfolio_s50_t5}
\begin{tabular}{cccccccc}
\toprule
\textbf{Rank} & \(\boldsymbol{\mu_B^0}\) & \(\boldsymbol{\mu_P^0}\) & \(\boldsymbol{\tau_M}\) & \textbf{LR} & \textbf{Mean \(e_P\)} & \textbf{Mean \(e_D\)} & \textbf{$e_{\mathrm{comb}}$} \\
\midrule
\(1\) & \(0.01\) & \(0.5\) & \(1\) & \(10^{-4}\) & \(0.0019\) & \(0.0092\) & \(0.0094\) \\
\(2\) & \(0.01\) & \(2\) & \(1\) & \(10^{-4}\) & \(0.0035\) & \(0.0110\) & \(0.0116\) \\
\(3\) & \(0.01\) & \(1\) & \(1\) & \(10^{-4}\) & \(0.0038\) & \(0.0120\) & \(0.0126\) \\
\(4\) & \(0.01\) & \(2\) & \(5\) & \(10^{-4}\) & \(0.0039\) & \(0.0134\) & \(0.0139\) \\
\(5\) & \(0.01\) & \(1\) & \(5\) & \(10^{-4}\) & \(0.0092\) & \(0.0225\) & \(0.0243\) \\
\bottomrule
\end{tabular}
\end{table}

\begin{table}[H]
\centering
\caption{Hyperparameter sensitivity for portfolio problems, \(s=200,\;t=20\).}
\label{tab:hparam_portfolio_s200_t20}
\begin{tabular}{cccccccc}
\toprule
\textbf{Rank} & \(\boldsymbol{\mu_B^0}\) & \(\boldsymbol{\mu_P^0}\) & \(\boldsymbol{\tau_M}\) & \textbf{LR} & \textbf{Mean \(e_P\)} & \textbf{Mean \(e_D\)} & \textbf{Combined} \\
\midrule
\(1\) & \(0.01\) & \(0.1\) & \(1\) & \(10^{-4}\) & \(0.0091\) & \(0.0988\) & \(0.0992\) \\
\(2\) & \(0.01\) & \(0.5\) & \(1\) & \(10^{-4}\) & \(0.0101\) & \(0.1170\) & \(0.1175\) \\
\(3\) & \(10^{-3}\) & \(5\times10^{3}\) & \(1\) & \(10^{-4}\) & \(0.0569\) & \(0.1688\) & \(0.1782\) \\
\(4\) & \(0.01\) & \(2\) & \(1\) & \(10^{-4}\) & \(0.0135\) & \(0.1800\) & \(0.1805\) \\
\(5\) & \(0.01\) & \(1\) & \(1\) & \(10^{-4}\) & \(0.0074\) & \(0.2256\) & \(0.2257\) \\
\bottomrule
\end{tabular}
\end{table}

\begin{table}[H]
\centering
\caption{Hyperparameter sensitivity for SVM problems, \(s=5,\;t=50\).}
\label{tab:hparam_svm_s5_t50}
\begin{tabular}{cccccc}
\toprule
\textbf{Rank} & \(\boldsymbol{\mu_B^0}\) & \(\boldsymbol{\mu_P^0}\) & \textbf{Mean \(e_P\)} & \textbf{Mean \(e_D\)} & \textbf{$e_{\mathrm{comb}}$} \\
\midrule
\(1\) & \(10^{-3}\) & \(0.5\) & \(0.0003\) & \(0.0086\) & \(0.0086\) \\
\(2\) & \(0.1\) & \(0.5\) & \(0.0005\) & \(0.0200\) & \(0.0200\) \\
\(3\) & \(0.01\) & \(2\) & \(0.0005\) & \(0.0312\) & \(0.0312\) \\
\(4\) & \(0.1\) & \(2\) & \(0.0006\) & \(0.0761\) & \(0.0761\) \\
\(5\) & \(10^{-3}\) & \(2\) & \(0.0023\) & \(0.0969\) & \(0.0969\) \\
\bottomrule
\end{tabular}
\end{table}

\begin{table}[H]
\centering
\caption{Hyperparameter sensitivity for SVM problems, \(s=20,\;t=200\).}
\label{tab:hparam_svm_s20_t200}
\begin{tabular}{cccccc}
\toprule
\textbf{Rank} & \(\boldsymbol{\mu_B^0}\) & \(\boldsymbol{\mu_P^0}\) & \textbf{Mean \(e_P\)} & \textbf{Mean \(e_D\)} & \textbf{$e_{\mathrm{comb}}$} \\
\midrule
\(1\) & \(10^{-3}\) & \(0.5\) & \(0.0002\) & \(0.0324\) & \(0.0324\) \\
\(2\) & \(0.01\) & \(0.5\) & \(0.0007\) & \(0.1345\) & \(0.1345\) \\
\(3\) & \(10^{-3}\) & \(5\) & \(0.0002\) & \(0.1726\) & \(0.1726\) \\
\(4\) & \(0.1\) & \(2\) & \(0.0039\) & \(0.1805\) & \(0.1805\) \\
\(5\) & \(10^{-3}\) & \(2\) & \(0.0065\) & \(0.3435\) & \(0.3436\) \\
\bottomrule
\end{tabular}
\end{table}

\begin{table}[H]
\centering
\caption{Hyperparameter sensitivity for single-quadrotor navigation,
$n=200$ and $T=11$.}
\label{tab:hparam_quadrotor_navigation_200}
\begin{tabular}{cccccc}
\toprule
\textbf{Rank}
& $\mu_B^0$
& $\mu_P^0$
& \textbf{Mean $e_P$}
& \textbf{Mean $e_D$}
& \textbf{$e_{\mathrm{comb}}$} \\
\midrule
1 & $10^{-2}$ & $0.1$
  & $4.0483\times10^{-5}$ & $0.0304$ & $0.0304$ \\
2 & $10^{-2}$ & $2$
  & $0.0096$ & $0.0414$ & $0.0426$ \\
3 & $10^{-3}$ & $2$
  & $0.0082$ & $0.0460$ & $0.0467$ \\
4 & $10^{-3}$ & $10$
  & $0.0299$ & $0.0811$ & $0.0865$ \\
5 & $10^{-2}$ & $10$
  & $0.0352$ & $0.1168$ & $0.1220$ \\
\bottomrule
\end{tabular}
\end{table}


\subsection{Further Details on Experiment Problem Formulations}

\subsubsection{Constrained Benchmark Problems}
\label{app:shared_constrained_benchmarks}

We follow the constrained benchmark construction used by
IPM-LSTM~\citep{gao2024ipm_lstm}, which in turn follows the problem
generation of~\citep{donti2021dc3}. The convex benchmark problems have the
form
\begin{equation}
\label{eq:shared_convex_qp}
\begin{aligned}
\min_{x\in\mathbb{R}^n}\quad
& \frac{1}{2}x^\top Q_0x+p_0^\top x \\
\text{s.t.}\quad
& p_j^\top x\leq q_j,
\qquad j=1,\ldots,l,\\
& p_j^\top x=q_j,
\qquad j=l+1,\ldots,m,\\
& x_i^L\leq x_i\leq x_i^U,
\qquad i=1,\ldots,n.
\end{aligned}
\end{equation}
The main shared benchmark uses \(n=200\) variables, \(100\) inequality
constraints, and \(100\) equality constraints. In the RHS setting, only the
right-hand sides of the equality constraints vary across instances, whereas
in the ALL setting, all problem parameters are perturbed.

Following the simple nonconvex benchmark of
IPM-LSTM~\citep{gao2024ipm_lstm}, we additionally replace the linear objective
term \(p_0^\top x\) with a sinusoidal term,
\[
\frac{1}{2}x^\top Q_0x+p_0^\top\sin(x),
\]
while retaining the same constraint structure and instance-generation
procedure. We denote the resulting problem classes by NC-RHS and NC-ALL.

\subsubsection{Box-Constrained QP Instance Generation}
\label{app:box_qp_generation}

For the dimension-scaling experiments, we generate convex box-constrained
quadratic programs of the form
\[
\min_x\;
\frac{1}{2}x^\top Qx+p^\top x
\qquad
\text{s.t.}
\qquad
\ell\leq x\leq u.
\]

For each variable \(x_j\), two values
\(a_j,b_j\sim\mathcal{U}(-1,1)\) are sampled independently, and the bounds
are defined as
\[
\ell_j=\min\{a_j,b_j\},
\qquad
u_j=\max\{a_j,b_j\}.
\]
A minimum interval width of \(10^{-4}\) is enforced by symmetrically
widening any smaller interval.For each coordinate with finite lower and upper bounds, the learned
optimizer is initialized at the midpoint of the box:
\[
x_j^{(0)}=\frac{\ell_j+u_j}{2}.
\]
Thus, the initial point satisfies the box constraints.

\subsubsection{Portfolio Optimization Problem}
\label{app:portfolio_formulation}

We follow the portfolio problem generation of
\citep{chen2024expressive}, with one modification: we use fully dense problem
matrices rather than the sparse construction used in the original
generation. Let \(z\in\mathbb{R}^s\) denote the asset weights and
\(d\in\mathbb{R}^t\) the factor variables. The portfolio problem is
\begin{equation}
\label{eq:portfolio_problem}
\begin{aligned}
\min_{z,d}\quad
& \frac{1}{2}z^\top Dz
+\frac{1}{2}d^\top d
-\mu^\top z \\
\text{s.t.}\quad
& d=Fz,\\
& \mathbf{1}^\top z=1,\\
& z\geq0.
\end{aligned}
\end{equation}
Thus, the problem is a convex quadratic program with linear equality
constraints and variable bounds.

We initialize the asset weights by drawing
\[
r_i\sim\mathcal{U}(0,1),\qquad
\widetilde z_i=r_i+10^{-3},
\]
and normalizing them:
\[
z_i^{(0)}
=
\frac{\widetilde z_i}
{\sum_{j=1}^{s}\widetilde z_j},
\qquad i=1,\ldots,s.
\]
The small offset \(10^{-3}\) prevents a sampled value from being
arbitrarily close to zero before normalization. This is useful for the
interior-point method, whose barrier terms and bound multipliers can
become very large near the bound \(z_i=0\). It is a numerical
initialization choice, not a minimum required asset weight after
normalization.

Since every \(\widetilde z_i>0\), the normalized weights satisfy
\(z_i^{(0)}>0\). Moreover,
\[
\sum_{i=1}^{s}z_i^{(0)}=1.
\]
Thus, the initial weights satisfy both the nonnegativity bounds and
the budget constraint: they allocate exactly 100\% of the available
budget.

We set the initial factor exposures to
\[
d^{(0)}=Fz^{(0)},
\]
which satisfies the other equality constraint, \(d-Fz=0\). Hence,
the initial primal point satisfies all portfolio constraints.

\subsubsection{Support Vector Machine Problem}
\label{app:svm_formulation}

We follow the SVM problem generation of
\citep{chen2024expressive}, again using fully dense problem matrices rather
than the sparse construction used in the original generation. Let
\(z\in\mathbb{R}^s\) denote the classifier variables and
\(\tau\in\mathbb{R}^t\) the slack variables. The SVM problem is
\begin{equation}
\label{eq:svm_problem}
\begin{aligned}
\min_{z,\tau}\quad
& \frac{1}{2}z^\top z
+\lambda\mathbf{1}^\top\tau \\
\text{s.t.}\quad
& \tau-\operatorname{Diag}(y)Dz-\mathbf{1}\geq0,\\
& \tau\geq0.
\end{aligned}
\end{equation}
This is a convex quadratic program with linear inequality constraints.

We initialize the feature vector and slack variables as
\[
z^{(0)}=\mathbf{0},
\qquad
\tau^{(0)}=(1+\delta)\mathbf{1},
\qquad
\delta=10^{-1}.
\]
With \(z^{(0)}=\mathbf{0}\), the SVM constraints require
\(\tau_i\geq 1\). Our choice \(\tau_i^{(0)}=1.1\) therefore gives
\[
\tau_i^{(0)}-y_iD_i z^{(0)}-1=\delta>0
\]
for every sample \(i\), while also satisfying \(\tau_i^{(0)}>0\).
The initial point is thus strictly inside the inequality-feasible
region, rather than lying exactly on its boundary. This positive
margin helps avoid excessively large interior-point terms at
initialization. The remaining inequality slacks and multipliers are
initialized from the constraint residuals and the initial penalty
and barrier parameters.

\subsubsection{Single-Quadrotor Navigation Problem}
\label{app:quadrotor_navigation}

We use a smaller randomized instance of the single-quadrotor navigation
problem from~\citep{viljoen2026scaling}. The objective, nonlinear
quadrotor dynamics, obstacle-avoidance constraint, and variable bounds
retain the same mathematical form. We modify the horizon and the
distribution of the problem parameters to obtain a visually meaningful
navigation problem with \(200\) optimization variables.

\paragraph{State and control variables.}
For \(T\) control intervals, let
\[
\mathbf{x}_k
=
\begin{bmatrix}
\mathbf{p}_k^\top &
\mathbf{q}_k^\top &
\mathbf{v}_k^\top &
\boldsymbol{\omega}_k^\top
\end{bmatrix}^{\top}
\in\mathbb{R}^{13},
\qquad k=0,\ldots,T,
\]
denote the quadrotor state. Its components are the position
\[
\mathbf{p}_k
=
\begin{bmatrix}
p_{x,k} & p_{y,k} & p_{z,k}
\end{bmatrix}^{\top}
\in\mathbb{R}^{3},
\]
the quaternion orientation
\[
\mathbf{q}_k
=
\begin{bmatrix}
q_{w,k} & q_{x,k} & q_{y,k} & q_{z,k}
\end{bmatrix}^{\top}
\in\mathbb{R}^{4},
\]
the linear velocity
\[
\mathbf{v}_k
=
\begin{bmatrix}
v_{x,k} & v_{y,k} & v_{z,k}
\end{bmatrix}^{\top}
\in\mathbb{R}^{3},
\]
and the angular velocity
\[
\boldsymbol{\omega}_k
=
\begin{bmatrix}
\omega_{x,k} & \omega_{y,k} & \omega_{z,k}
\end{bmatrix}^{\top}
\in\mathbb{R}^{3}.
\]
Thus, each state vector contains \(13\) variables.

The control vector is
\[
\mathbf{u}_k
=
\begin{bmatrix}
u_{k,1} & u_{k,2} & u_{k,3} & u_{k,4}
\end{bmatrix}^{\top}
\in\mathbb{R}^{4},
\qquad k=0,\ldots,T-1,
\]
where \(u_{k,j}\) is the angular speed of rotor \(j\) during control
interval \(k\).

There are \(T+1\) state nodes,
\(\mathbf{x}_0,\ldots,\mathbf{x}_T\), but only \(T\) control vectors,
\(\mathbf{u}_0,\ldots,\mathbf{u}_{T-1}\), because each control vector
governs the transition from \(\mathbf{x}_k\) to
\(\mathbf{x}_{k+1}\). The complete decision vector is therefore
\[
\boldsymbol{\xi}
=
\begin{bmatrix}
\mathbf{x}_0^\top &
\cdots &
\mathbf{x}_T^\top &
\mathbf{u}_0^\top &
\cdots &
\mathbf{u}_{T-1}^\top
\end{bmatrix}^{\top}
\in\mathbb{R}^{n},
\]
with
\[
n
=
13(T+1)+4T
=
13+17T.
\]

\paragraph{Equality constraints: quadrotor dynamics.}
The continuous-time dynamics describe how the quadrotor's position,
orientation, linear velocity, and angular velocity change under the
four rotor inputs:
\[
f_{\mathrm{dyn}}(\mathbf{x},\mathbf{u})
=
\dot{\mathbf{x}}
=
\begin{bmatrix}
\dot{\mathbf{p}}
\\[0.2em]
\dot{\mathbf{q}}
\\[0.2em]
\dot{\mathbf{v}}
\\[0.2em]
\dot{\boldsymbol{\omega}}
\end{bmatrix}
=
\begin{bmatrix}
\mathbf{v}
\\[0.4em]
\displaystyle
\frac{1}{2}\,
\mathbf{q}\otimes
\begin{bmatrix}
0\\
\boldsymbol{\omega}
\end{bmatrix}
\\[0.8em]
\displaystyle
\frac{
\mathbf{f}_{\mathrm{drag}}
+
\mathbf{f}_{\mathrm{thrust}}
+
\mathbf{f}_{\mathrm{gravity}}
}{m_B}
\\[1em]
\displaystyle
I_B^{-1}
\left(
\boldsymbol{\tau}_{\mathrm{gyro}}
+
\boldsymbol{\tau}_{\mathrm{rotor}}
+
\boldsymbol{\tau}_{\mathrm{prec}}
\right)
\end{bmatrix}.
\]
Here, \(m_B\) is the quadrotor mass, \(I_B\) is its body-inertia
matrix, and \(\otimes\) denotes quaternion multiplication.
The force and torque terms represent aerodynamic drag, rotor thrust,
gravity, rigid-body gyroscopic torque, rotor-induced torque, and rotor
precession, respectively.

The four blocks of the dynamics have different constraint structures.
The relation
\(\dot{\mathbf{p}}=\mathbf{v}\) is linear, whereas the orientation,
linear-velocity, and angular-velocity dynamics contain products of state
variables, orientation-dependent thrust, squared rotor speeds, and
precession terms and are therefore nonlinear and nonconvex.

We impose the continuous dynamics using explicit Euler discretization:
\[
\mathbf{x}_{k+1}
-
\mathbf{x}_k
-
\Delta t\,
f_{\mathrm{dyn}}(\mathbf{x}_k,\mathbf{u}_k)
=
\mathbf{0},
\qquad k=0,\ldots,T-1.
\]
Each control interval therefore contributes \(13\) equality constraints:
three linear position-update constraints and ten nonlinear, nonconvex
orientation, linear-velocity, and angular-velocity constraints.

\paragraph{Inequality and bound constraints: obstacle avoidance and variable bounds.}
Obstacle avoidance is imposed using the horizontal position
\[
\mathbf{p}_{k,xy}
=
\begin{bmatrix}
p_{x,k} & p_{y,k}
\end{bmatrix}^{\top}.
\]
Because the collision constraint depends only on \(p_x\) and \(p_y\), the
obstacle represents a vertical cylinder. The quadrotor is represented by
its position point, and collision avoidance requires this point to remain
outside a disk centered at the obstacle with effective collision radius
\[
\rho_{\mathrm{eff}}
=
r_{\mathrm{obs}}+r_{\mathrm{quad}},
\]
which accounts for both the obstacle radius and the physical size of the
quadrotor. The required squared clearance is slightly increased along the
trajectory by the factor
\[
1+0.1k\Delta t.
\]

The finite state bounds are
\[
-10
\leq
p_{x,k},p_{y,k},p_{z,k},
v_{x,k},v_{y,k},v_{z,k},
\omega_{x,k},\omega_{y,k},\omega_{z,k}
\leq
10.
\]
The quaternion components have no additional finite box bounds. The
rotor-speed bounds are
\[
75
\leq
u_{k,j}
\leq
925,
\qquad j=1,\ldots,4.
\]

\paragraph{Instance-dependent quantities.}
Each problem instance is parameterized by
\[
\psi=(\theta,R,\alpha,\beta),
\]
where \(\theta\) and \(R\) determine the initial state and
\(\alpha\) and \(\beta\) determine the obstacle location relative to the
initial position. We denote the corresponding initial state by
\(\mathbf{x}_{\mathrm{init}}(\theta,R)\) and the obstacle center by
\(\mathbf{c}_{\mathrm{obs}}(\theta,R,\alpha,\beta)\). 

\paragraph{Optimization problem.}
With the state, dynamics, obstacle geometry, and instance parameters defined
above, the quadrotor navigation problem is
\begin{equation}
\label{eq:quadrotor_navigation_problem}
\begin{aligned}
\underset{\boldsymbol{\xi}}{\operatorname{minimize}}
\quad &
\sum_{k=0}^{T}
\left(
\|\mathbf{p}_k\|_2^2
+
\|\mathbf{v}_k\|_2^2
+
\|\boldsymbol{\omega}_k\|_2^2
\right)
\\
\operatorname{subject\ to}
\quad &
\mathbf{x}_0
=
\mathbf{x}_{\mathrm{init}}(\theta,R),
\\
&
\mathbf{x}_{k+1}
=
\mathbf{x}_k
+
\Delta t\,
f_{\mathrm{dyn}}(\mathbf{x}_k,\mathbf{u}_k),
&& k=0,\ldots,T-1,
\\
&
\left\|
\mathbf{p}_{k,xy}
-
\mathbf{c}_{\mathrm{obs}}(\theta,R,\alpha,\beta)
\right\|_2^2
-
\rho_{\mathrm{eff}}^2
\left(1+0.1k\Delta t\right)
\geq 0,
&& k=0,\ldots,T-1,
\\
&
\underline{\mathbf{x}}
\leq
\mathbf{x}_k
\leq
\overline{\mathbf{x}},
&& k=0,\ldots,T,
\\
&
\underline{\mathbf{u}}
\leq
\mathbf{u}_k
\leq
\overline{\mathbf{u}},
&& k=0,\ldots,T-1.
\end{aligned}
\end{equation}

The objective penalizes the quadrotor's distance from the origin, linear
velocity, and angular velocity throughout the trajectory, encouraging it to
approach the goal at the origin and come to rest. The initial-state equality
fixes the first state of the trajectory, the dynamics equalities require
successive states to obey the discretized quadrotor dynamics, and the
obstacle-avoidance inequalities enforce collision-free horizontal motion.

For \(T=11\), there are \(T+1=12\) state nodes and \(T=11\) control
vectors, giving
\[
n
=
13(12)+4(11)
=
200
\]
optimization variables. The initial-state condition contributes \(13\)
linear equalities, while the dynamics contribute \(3T=33\) linear and
\(10T=110\) nonlinear equalities. Thus, the problem contains \(46\) linear
equality constraints and \(110\) nonlinear equality constraints, for a total
of
\[
m_{\mathrm{eq}}
=
13+13T
=
156
\]
equality constraints. It also contains
\[
m_{\mathrm{ineq}}
=
T
=
11
\]
nonlinear obstacle-avoidance constraints, in addition to the variable box
bounds.

\paragraph{Random instance generation.}
Define the radial and tangential unit vectors
\[
\mathbf{e}_r(\theta)
=
\begin{bmatrix}
\cos\theta\\
\sin\theta
\end{bmatrix},
\qquad
\mathbf{e}_t(\theta)
=
\begin{bmatrix}
-\sin\theta\\
\cos\theta
\end{bmatrix}.
\]
For each dataset instance \(i\), we independently sample
\[
\theta_i
\sim
\mathcal{U}\left(0,\frac{\pi}{2}\right),
\]
\[
R_i
\sim
\mathcal{U}(0.285,0.315)\ \mathrm{m},
\]
\[
\alpha_i
\sim
\mathcal{U}(0.47,0.53),
\qquad
\beta_i
\sim
\mathcal{U}(-0.07,-0.04)\ \mathrm{m}.
\]

The initial state is
\[
\mathbf{x}_{\mathrm{init}}^{(i)}
=
\begin{bmatrix}
R_i\cos\theta_i\\
R_i\sin\theta_i\\
0\\
1\\
0\\
0\\
0\\
-\,v_r\cos\theta_i-v_t\sin\theta_i\\
-\,v_r\sin\theta_i+v_t\cos\theta_i\\
0\\
0\\
0\\
0
\end{bmatrix},
\]
where
\[
v_r=0.30\ \mathrm{m/s},
\qquad
v_t=0.18\ \mathrm{m/s}.
\]
Equivalently, the initial horizontal velocity is
\[
\mathbf{v}_{0,xy}^{(i)}
=
-v_r\mathbf{e}_r(\theta_i)
+
v_t\mathbf{e}_t(\theta_i).
\]

The obstacle center is
\[
\mathbf{c}_{\mathrm{obs}}^{(i)}
=
\alpha_iR_i\mathbf{e}_r(\theta_i)
+
\beta_i\mathbf{e}_t(\theta_i).
\]
The parameter \(\alpha_i\) determines how far the obstacle lies along
the route from the initial position to the goal, while \(\beta_i\)
determines its lateral displacement from the direct route. Although the
obstacle center is expressed in the local radial--tangential frame associated
with the initial position, \(\alpha_i\) and \(\beta_i\) are sampled
independently of \(\theta_i\) and \(R_i\). This construction produces
different initial positions and obstacle locations while keeping the obstacle
relevant to the navigation task.

Across all instances, we fix
\[
T=11,
\qquad
\Delta t=0.35\ \mathrm{s},
\qquad
v_r=0.30\ \mathrm{m/s},
\qquad
v_t=0.18\ \mathrm{m/s},
\]
and
\[
r_{\mathrm{obs}}=0.02\ \mathrm{m},
\qquad
r_{\mathrm{quad}}=0.08\ \mathrm{m},
\qquad
\rho_{\mathrm{eff}}=0.10\ \mathrm{m}.
\]
The randomized quantities are
\[
\theta_i,\qquad R_i,\qquad\alpha_i,\qquad\beta_i,
\]
which determine the initial position, initial velocity direction, and
obstacle center of each problem instance.

\subsection{Additional Experimental Results}
\label{app:additional-experiments}

\subsubsection{Warm-Starting Results}
\label{app:additional-warm-start}

The following tables provide additional solver-level statistics for the
warm-start experiments reported in the main text, including gradient and
function evaluation counts for cold-started pdProj, warm-started pdProj,
and IPOPT. All reported iteration counts, evaluation counts, and runtimes
are averages per problem over the held-out test set. Full convergence is
assessed for all solvers using the solver-neutral criterion
\(r_{\mathrm{KKT}}\leq10^{-8}\) defined in
Appendix~\ref{app:kkt_residual}. In the runtime table, WS Cost denotes the
learned warm-start generation time, and Total Time includes both WS Cost
and the subsequent pdProj solve.

\paragraph{Constrained QPs at \(n=200\).}
\leavevmode\par

\begin{table}[H]
\centering
\caption{Iterations, gradient calls, and function calls for the \(n=200\) constrained benchmarks.}
\label{tab:constrained_qp_eval_counts_appendix_200dim}

\begin{tabular}{llccc}
\toprule
\textbf{Problem Class}
& \textbf{Method}
& \textbf{Iter.}
& \textbf{Grad. Calls}
& \textbf{Func. Calls} \\
\midrule

Convex QP RHS
& IPOPT
& \(15.63\)
& \(17.63\)
& \(16.63\) \\

& Cold pdProj
& \(13.41\)
& \(16.73\)
& \(16.73\) \\

& Warm-started pdProj
& \(4.64\)
& \(5.64\)
& \(5.64\) \\

\midrule

Convex QP ALL
& IPOPT
& \(15.82\)
& \(17.82\)
& \(16.82\) \\

& Cold pdProj
& \(13.22\)
& \(17.65\)
& \(17.65\) \\

& Warm-started pdProj
& \(4.35\)
& \(5.36\)
& \(5.36\) \\

\midrule

Nonconvex RHS
& IPOPT
& \(15.60\)
& \(17.60\)
& \(16.60\) \\

& Cold pdProj
& \(13.47\)
& \(16.88\)
& \(16.88\) \\

& Warm-started pdProj
& \(4.93\)
& \(5.95\)
& \(5.95\) \\

\midrule

Nonconvex ALL
& IPOPT
& \(15.72\)
& \(17.72\)
& \(16.72\) \\

& Cold pdProj
& \(13.46\)
& \(17.84\)
& \(17.84\) \\

& Warm-started pdProj
& \(4.84\)
& \(5.87\)
& \(5.87\) \\

\bottomrule
\end{tabular}
\end{table}

\paragraph{Constrained QPs at \(n=50\).}
\leavevmode\par

\begin{table}[H]
\centering
\caption{Iterations, gradient calls, and function calls for the \(n=50\) constrained benchmarks.}
\label{tab:constrained_qp_iteration_eval_counts_appendix_50dim}

\begin{tabular}{llccc}
\toprule
\textbf{Problem Class}
& \textbf{Method}
& \textbf{Iter.}
& \textbf{Grad. Calls}
& \textbf{Func. Calls} \\
\midrule

Convex QP RHS
& IPOPT
& \(12.08\)
& \(14.08\)
& \(13.08\) \\

& Cold pdProj
& \(8.97\)
& \(10.38\)
& \(10.38\) \\

& Warm-started pdProj
& \(2.92\)
& \(3.93\)
& \(3.93\) \\

& Iteration reduction
& \(\mathbf{67.49\%}\)
& --
& -- \\

\midrule

Convex QP ALL
& IPOPT
& \(12.17\)
& \(14.17\)
& \(13.17\) \\

& Cold pdProj
& \(9.31\)
& \(11.00\)
& \(11.00\) \\

& Warm-started pdProj
& \(2.85\)
& \(3.85\)
& \(3.85\) \\

& Iteration reduction
& \(\mathbf{69.41\%}\)
& --
& -- \\

\midrule

Nonconvex RHS
& IPOPT
& \(12.04\)
& \(14.04\)
& \(13.04\) \\

& Cold pdProj
& \(9.39\)
& \(10.76\)
& \(10.76\) \\

& Warm-started pdProj
& \(2.95\)
& \(3.96\)
& \(3.96\) \\

& Iteration reduction
& \(\mathbf{68.57\%}\)
& --
& -- \\

\midrule

Nonconvex ALL
& IPOPT
& \(12.17\)
& \(14.17\)
& \(13.17\) \\

& Cold pdProj
& \(9.51\)
& \(11.39\)
& \(11.39\) \\

& Warm-started pdProj
& \(3.05\)
& \(4.09\)
& \(4.09\) \\

& Iteration reduction
& \(\mathbf{67.95\%}\)
& --
& -- \\

\bottomrule
\end{tabular}
\end{table}

\begin{table}[H]
\centering
\caption{Runtime results for the \(n=50\) constrained benchmarks.}
\label{tab:constrained_qp_time_reductions_appendix_50dim}

\begin{tabular}{lcccc}
\toprule
\textbf{Problem Class}
& \textbf{Cold Time}
& \textbf{WS Cost}
& \textbf{Total Time}
& \textbf{Time Red.} \\
\midrule

Convex QP RHS
& \(2.69\) s
& \(5.23\) ms
& \(878\) ms
& \(\mathbf{67.31\%}\) \\

Convex QP ALL
& \(2.80\) s
& \(4.59\) ms
& \(861\) ms
& \(\mathbf{69.22\%}\) \\

Nonconvex RHS
& \(2.81\) s
& \(5.31\) ms
& \(886\) ms
& \(\mathbf{68.44\%}\) \\

Nonconvex ALL
& \(2.84\) s
& \(5.64\) ms
& \(917\) ms
& \(\mathbf{67.75\%}\) \\

\bottomrule
\end{tabular}
\end{table}

\paragraph{Convex box-constrained QPs.}
\leavevmode\par

\begin{table}[H]
\centering
\caption{Iterations, gradient calls, and function calls for the box-constrained QPs.}
\label{tab:qpbc_eval_counts}

\begin{tabular}{llccc}
\toprule
\textbf{Problem}
& \textbf{Method}
& \textbf{Iter.}
& \textbf{Grad. Calls}
& \textbf{Func. Calls} \\
\midrule

\(n=200\)
& IPOPT
& \(13.48\)
& \(15.48\)
& \(14.48\) \\

& Cold pdProj
& \(11.44\)
& \(12.44\)
& \(12.44\) \\

& Warm-started pdProj
& \(2.19\)
& \(3.19\)
& \(3.19\) \\

\midrule

\(n=1000\)
& IPOPT
& \(15.50\)
& \(17.50\)
& \(16.50\) \\

& Cold pdProj
& \(13.87\)
& \(14.87\)
& \(14.87\) \\

& Warm-started pdProj
& \(7.23\)
& \(8.23\)
& \(8.23\) \\

\bottomrule
\end{tabular}
\end{table}

\paragraph{Portfolio and SVM problems.}
\leavevmode\par

\begin{table}[H]
\centering
\caption{Iterations, gradient calls, and function calls for the portfolio and SVM benchmarks.}
\label{tab:portfolio_svm_eval_counts_appendix}

\begin{tabular}{lllccc}
\toprule
\textbf{Problem}
& \textbf{$(s,t)$}
& \textbf{Method}
& \textbf{Iter.}
& \textbf{Grad. Calls}
& \textbf{Func. Calls} \\
\midrule

Portfolio
& $(50,5)$
& IPOPT
& $14.68$
& $16.68$
& $15.68$ \\

&
& Cold pdProj
& $8.99$
& $9.99$
& $9.99$ \\

&
& Warm pdProj
& $3.48$
& $4.48$
& $4.48$ \\

\midrule

Portfolio
& $(200,20)$
& IPOPT
& $21.06$
& $23.06$
& $22.06$ \\

&
& Cold pdProj
& $12.07$
& $13.18$
& $13.18$ \\

&
& Warm pdProj
& $9.36$
& $10.57$
& $10.57$ \\

\midrule

SVM
& $(5,50)$
& IPOPT
& $12.52$
& $14.52$
& $13.53$ \\

&
& Cold pdProj
& $9.30$
& $10.33$
& $10.33$ \\

&
& Warm pdProj
& $4.52$
& $5.56$
& $5.56$ \\

\midrule

SVM
& $(20,200)$
& IPOPT
& $15.46$
& $17.46$
& $16.46$ \\

&
& Cold pdProj
& $13.83$
& $15.01$
& $15.01$ \\

&
& Warm pdProj
& $6.31$
& $7.31$
& $7.31$ \\

\bottomrule
\end{tabular}
\end{table}

\paragraph{Quadrotor navigation problem.}
\leavevmode\par

\begin{table}[H]
\centering
\caption{Iterations, gradient calls, and function calls for the nonlinear quadrotor control problem.}
\label{tab:quadrotor_eval_counts_appendix}

\begin{tabular}{llccc}
\toprule
\textbf{Problem}
& \textbf{Method}
& \textbf{Iter.}
& \textbf{Grad. Calls}
& \textbf{Func. Calls} \\
\midrule

Quadrotor
& IPOPT
& \(9.20\)
& \(11.20\)
& \(10.20\) \\

& Cold pdProj
& \(9.74\)
& \(10.74\)
& \(10.74\) \\

& Warm-started pdProj
& \(8.00\)
& \(9.00\)
& \(9.00\) \\

\bottomrule
\end{tabular}
\end{table}

\subsubsection{Direct Approximate-Solution Comparison}
\label{app:direct-approx}

In addition to evaluating the model as a warm-start generator, we evaluate
the learned optimizer directly at a lower-accuracy tolerance. These
experiments test whether the learned iterations can produce useful
approximate solutions without subsequent pdProj refinement. For both the
learned optimizer and IPOPT, convergence is assessed using the same
solver-neutral criterion \(r_{\mathrm{KKT}}\leq10^{-2}\). We report the
fraction of problems satisfying this criterion and the average runtime per
problem. The learned optimizer is run on an NVIDIA RTX PRO 6000 Blackwell
Server Edition GPU using Google Colab, while IPOPT is run on CPU.

\paragraph{Convex and constrained QPs at \(n=200\) and \(n=1000\).}
\leavevmode\par

\begin{table}[H]
\centering
\caption{Direct approximate-solution results on the larger QP benchmarks.}
\label{tab:direct_solver_comparison_qps}

\begin{tabular}{llcc}
\toprule
\textbf{Problem}
& \textbf{Method}
& \textbf{Converged}
& \textbf{Avg. Runtime} \\
\midrule

Box QP, \(n=200\)
& pdLIP
& \(100\%\)
& \(4.02\) ms \\

& IPOPT
& \(100\%\)
& \(85.27\) ms \\

\midrule

Constrained QP RHS, \(n=200\)
& pdLIP
& \(72.4\%\)
& \(20.11\) ms \\

& IPOPT
& \(100\%\)
& \(315.54\) ms \\

\midrule

Constrained QP ALL, \(n=200\)
& pdLIP
& \(85.2\%\)
& \(19.93\) ms \\

& IPOPT
& \(100\%\)
& \(344.84\) ms \\

\midrule

Box QP, \(n=1000\)
& pdLIP
& \(78\%\)
& \(43.61\) ms \\

& IPOPT
& \(100\%\)
& \(1.47\) s \\

\bottomrule
\end{tabular}
\end{table}

Table~\ref{tab:direct_solver_comparison_qps} shows that the learned optimizer
can produce approximate solutions substantially faster than IPOPT on these
convex QP families. On the \(200\)-dimensional problems, the learned
optimizer reaches the solver-neutral KKT tolerance in milliseconds, while IPOPT
requires longer average runtimes. The learned optimizer does not solve every
constrained instance to the prescribed tolerance, but it obtains high
convergence rates in a substantially shorter time.

The larger box-constrained QP experiment provides an additional
dimension-scaling comparison. At \(n=1000\), the learned optimizer reaches
the prescribed tolerance on \(78\%\) of test problems in \(43.61\) ms on
average, compared with \(1.47\) s for IPOPT. This behavior is consistent
with the method design: during learned iterations, the curvature-dependent
Hessian construction and linear-system solve are replaced by the learned
coordinate-wise scaling.

\paragraph{Constrained QPs at \(n=50\).}
\leavevmode\par

\begin{table}[H]
\centering
\caption{Direct approximate-solution results on the \(n=50\) constrained benchmarks.}
\label{tab:qp_direct_solver_comparison}

\begin{tabular}{llcc}
\toprule
\textbf{Problem}
& \textbf{Method}
& \textbf{Converged}
& \textbf{Avg. Runtime} \\
\midrule

Convex RHS
& pdLIP
& \(95.3\%\)
& \(5.23\) ms \\

& IPOPT
& \(100\%\)
& \(4.29\) ms \\

\midrule

Convex ALL
& pdLIP
& \(91.7\%\)
& \(4.59\) ms \\

& IPOPT
& \(100\%\)
& \(4.43\) ms \\

\midrule

Nonconvex RHS
& pdLIP
& \(92.5\%\)
& \(5.31\) ms \\

& IPOPT
& \(100\%\)
& \(4.37\) ms \\

\midrule

Nonconvex ALL
& pdLIP
& \(75.0\%\)
& \(5.64\) ms \\

& IPOPT
& \(100\%\)
& \(4.56\) ms \\

\bottomrule
\end{tabular}
\end{table}

\end{document}